\documentclass[a4paper,11pt]{article}

\pdfoutput=1

\usepackage[unicode, colorlinks,citecolor=blue,urlcolor=blue]{hyperref}
\usepackage{mathtools}
\usepackage{authblk}
\usepackage[round, authoryear]{natbib}
\usepackage[utf8]{inputenc} 
\usepackage[T1]{fontenc}    
\usepackage{lmodern}
\usepackage{amsthm}
\usepackage{amsmath}
\usepackage{amssymb}
\usepackage{url}            
\usepackage{booktabs}       
\usepackage{amsfonts}       
\usepackage{mathtools}      
\usepackage{nicefrac}       
\usepackage{microtype}      
\usepackage{subfiles}       
\usepackage{bbold}          
\usepackage{xcolor}         
\usepackage{bm}             
\usepackage[font=footnotesize]{subcaption}     
\usepackage[font=footnotesize]{caption}
\usepackage{enumitem}
\usepackage[margin=3cm]{geometry}
\usepackage{empheq}
\usepackage{framed}
\usepackage{fullpage}

\numberwithin{equation}{section}
\theoremstyle{plain}

\newtheorem{theorem}{Theorem}[section]
\newtheorem{proposition}{Proposition}[section]
\newtheorem{lemma}[theorem]{Lemma}
\newtheorem{corollary}[theorem]{Corollary}
\newtheorem{definition}[theorem]{Definition}
\theoremstyle{remark}
\newtheorem{remark}[theorem]{Remark}

\mathtoolsset{showonlyrefs}
\renewcommand{\paragraph}[1]{{\textbf{#1}}}
\usepackage{macros}

\title{Exponential Tail Local Rademacher Complexity Risk Bounds Without
  the Bernstein Condition}
\author[1]{Varun Kanade}
\author[2]{Patrick Rebeschini}
\author[2]{Tomas Va\v{s}kevi\v{c}ius}
\affil[1]{Department of Computer Science, University of Oxford}
\affil[2]{Department of Statistics, University of Oxford}

\begin{document}

\maketitle

\begin{abstract}
The local Rademacher complexity framework is one of the most successful
general-purpose toolboxes for establishing sharp excess risk bounds for
statistical estimators based on the framework of empirical risk minimization.
Applying this toolbox typically requires using the Bernstein condition, which
often restricts applicability to convex and proper settings. Recent years have
witnessed several examples of problems where optimal statistical performance is
only achievable via non-convex and improper estimators originating from
aggregation theory, including the fundamental problem of model selection. These
examples are currently outside of the reach of the classical localization
theory.

In this work, we build upon the recent approach to localization via offset
Rademacher complexities, for which a general high-probability theory has yet to
be established. Our main result is an exponential-tail excess risk bound
expressed in terms of the offset Rademacher complexity that yields results at
least as sharp as those obtainable via the classical theory. However, our bound
applies under an estimator-dependent geometric condition (the "offset
condition") instead of the estimator-independent (but, in general,
distribution-dependent) Bernstein condition on which the classical theory
relies. Our results apply to improper prediction regimes not directly covered
by the classical theory.
\end{abstract}


\section{Introduction}
\label{sec:introduction}
We study the problem of obtaining performance estimates on a general class of
statistical prediction procedures.  Let $S_{n} = (X_{i}, Y_{i})_{i=1}^{n}$
denote an i.i.d.\ sample of input-output pairs $(X_{i}, Y_{i}) \in \mathcal{X}
\times \mathcal{Y}$ distributed according to some \emph{unknown} distribution
$P$.  A function mapping $\mathcal{X}$ to $\mathcal{Y}$ is called a
\emph{predictor}. A \emph{statistical estimator} is a procedure mapping the
observed random sample $S_{n}$ to some predictor $\widehat{f} =
\widehat{f}(S_{n}) \in \mathcal{F}$, where the class $\mathcal{F}$ is called
the \emph{range} of the estimator $\widehat{f}$.  In order to measure the
quality of an estimator $\widehat{f}$, we introduce a \emph{loss function}
$\ell : \mathcal{Y} \times \mathcal{Y} \to [0, \infty)$ and define the
performance measure called \emph{risk} as follows:
\begin{equation}
  \label{eq:risk-definition}
  R(\widehat{f}) = \E_{(X,Y) \sim P}[\ell(\widehat{f}(X), Y) \vert S_{n}].
\end{equation}
The above performance measure is absolute, and its scale depends on the
properties of the loss function $\ell$ as well as the distribution $P$.  In
order to obtain a performance measure whose value can approach zero as the
sample size $n$ increases, it is customary to introduce a class of
\emph{reference predictors} $\mathcal{G}$.  The risk incurred by the estimator
$\widehat{f}$, relative to the smallest risk achievable via predictors in class
$\mathcal{G}$, is called \emph{excess risk} and it is defined as
\begin{equation}
  \label{eq:excess-risk-definition}
  \mathcal{E}_{P}(\widehat{f}, \mathcal{G})
  = R(\widehat{f}) - \inf_{g \in \mathcal{G}} R(g).
\end{equation}
Observe that we have not imposed any restrictions on the distribution $P$,
other than constraining it to be supported on $\mathcal{X} \times \mathcal{Y}$.
Such a setting is called \emph{agnostic}, \emph{distribution-free} or
\emph{misspecified}, and it has been a central object of study in Statistical
Learning Theory since the early works of \citet*{vapnik1968uniform,
vapnik1971uniform, vapnik1974theory}.  This setup should be contrasted with the
\emph{well-specified} setting, where the reference class of functions
$\mathcal{G}$ is taken to be $\mathcal{F}$, the range of the estimator
$\widehat{f}$, and the observations are assumed to follow the distribution
$Y_{i} = f(X_{i}) + \xi_{i}$ for some $f \in \mathcal{F}$ and zero-mean noise
$\xi_{i}$. The present paper focuses on obtaining excess risk bounds that hold
for \emph{any} distribution $P$ supported on $\mathcal{X} \times \mathcal{Y}$; that
is, we study the distribution-free setting.

Upper bounds on the excess risk $\mathcal{E}_{P}(\widehat{f}, \mathcal{G})$ can
be obtained either in \emph{expectation} or in \emph{deviation}. The former
type of bounds aims to find the smallest remainder term $\Delta_{\mathbf{E}}(n,
\mathcal{G})$ that depends on properties of the estimator $\widehat{f}$ such as
its range $\mathcal{F}$ so that for some universal constant $c > 0$ the
following holds:
\begin{equation}
  \label{eq:expected-bounds}
  \E_{S_{n}}\big[\mathcal{E}_{P}(\widehat{f}, \mathcal{G})\big]
  \leq c \Delta_{\mathbf{E}}(n, \mathcal{G}).
\end{equation}

Similarly, bounds in deviation aim to find the smallest remainder term
$\Delta_{\mathbf{Pr}}$ that depends on properties of the estimator $\widehat{f}$ so that
the following holds for any $\delta \in (0,1]$:
\begin{equation}
  \label{eq:deviation-bounds}
  \P_{S_{n}}\big(
    \mathcal{E}_{P}(\widehat{f}, \mathcal{G})
    > c'\Delta_{\mathbf{Pr}}(n, \mathcal{G}, \delta)
  \big)
  \leq \delta,
\end{equation}
where $c' > 0$ is some universal constant. Observe that bounds of the above
type can be transformed to in-expectation bounds via tail integration
arguments; hence, obtaining sharp excess risk bounds that hold with
high probability is typically a more challenging problem than obtaining
in-expectation guarantees. If the remainder term
$\Delta_{\mathbf{Pr}}(n, \mathcal{G}, \delta)$ is of order $\log(1/\delta)$ as
a function of $\delta$, we call such guarantees \emph{exponential tail}
bounds.

Several frameworks have been developed for obtaining both types of statistical
performance guarantees. One of the simplest ways to obtain sharp in-expectation
guarantees without imposing strong distributional assumptions is via
\emph{average stability} (or \emph{leave-one-out}) arguments
\citep*{rogers1978finite, devroye1979distribution, haussler1994predicting}.
Among other approaches are in-expectation guarantees obtainable via stochastic
approximation arguments
(e.g., \citep*{robbins1951stochastic, walk1989convergence,
nemirovski2009robust, dieuleveut2016nonparametric}),
or by transporting regret bounds from the framework of prediction of individual sequences
\citep*{cesa2006prediction} to the stochastic setting via an online-to-batch
conversion (e.g., \citep*{cesa2004generalization, audibert2009fast}).

Recently, there has been a growing interest in obtaining sharp excess risk
bounds that hold with high probability.
One challenge in converting in-expectation guarantees to in-deviation
counterparts is that, typically, simply applying concentration tools results in extra
deviation terms of order $n^{-1/2}$.
Consequently, stochastic conversions of
``fast rate'' in-expectation guarantees of order $n^{-1}$ are converted to
in-deviation guarantees with the ``slow rate'' $n^{-1/2}$.  To preserve
optimal rates, stochastic conversions need to be performed via
probabilistic tools capable of taking some notion of variance into account
(e.g., Bernstein's inequality) while, \emph{at the same time}, extinguishing
the resulting variance terms by exploiting curvature of the loss function,
or imposed ``niceness'' (e.g., low noise) assumptions on the underlying
data-generating distribution.
While this conversion has been carried out successfully in a few important
cases of interest, as we are going to describe below, the wide applicability of
this machinery is limited as typically either the variance terms are too large
or because properly bounding them comes at the price of introducing restrictive
assumptions.

For the class of \emph{uniformly stable} algorithms (which is a more
restrictive notion than average stability; see the work by \citet*{bousquet2002stability}),
``fast rate'' excess risk bounds that hold with high-probability were recently
obtained by \citet*{klochkov2021stability}, while for online-to-batch
conversions see the work by \citet*{kakade2009generalization} and the references therein. In
terms of probabilistic tools, the former work builds on the notion of (weakly)
self-bounding functions \citep*{boucheron2000sharp, maurer2006concentration},
while the latter relies on the tail bound for martingales due to
\citet*{freedman1975tail}.
However, both works cited above impose strong assumptions on the loss
function -- assumptions that we will not use in the theory we are going to develop
in this paper. These assumptions are typically not satisfied in classical
settings of interest, such as in the case of regression with the squared loss.
Specifically, these works assume that the loss function is strongly convex
when the domain of the loss function is taken to be the parameter space
of the predictors. For example, in the setting of linear regression with
quadratic loss, such an assumption would amount to restrictions on the
smallest eigenvalue of the empirical covariance matrix.

One of the most successful general-purpose tool for obtaining sharp excess
risk upper bounds is the \emph{local Rademacher complexity}
\citep*{bartlett2005local, koltchinskii2006local, koltchinskii2011oracle}.
This approach automatically comes with exponential-tail in-deviation guarantees
due to the underlying mathematical machinery resting on a powerful
concentration bound for controlling the supremum of empirical processes due to
\citet*{talagrand1994sharper, talagrand1996new}.
At the same time, (localized) Rademacher averages are relatively simple
to upper bound, with many settings of interest covered in the existing
literature; for some examples, see the textbook by
\citet*[Chapters 13 and 14]{wainwright2019high}.

Due to technical reasons related to the so-called Bernstein condition
(see Section~\ref{sec:background-classical-localization} for a detailed
discussion), local Rademacher
complexity bounds are primarily suitable when two conditions hold:
$\mathcal{G}$ is convex and $\mathcal{F} = \mathcal{G}$.  A setup when
$\mathcal{F} = \mathcal{G}$ is called \emph{proper}.  Soon after the
development of local Rademacher complexities, it was noticed in the discussion
paper by \citet*{tsybakov-rademacher-discussion} that such restrictions fail to
include a very natural problem called \emph{model selection aggregation}
\citep*{nemirovski2000topics, tsybakov2003optimal}. In this problem,
the reference class of functions $\mathcal{G}$ is taken to be a
finite set of bounded functions; particularly, it is a non-convex set, and
local Rademacher complexity theory does not apply directly.
Understanding how to optimally aggregate statistical models constructed from
i.i.d.\ data (e.g., models arising from different tuning parameters, or
different statistical estimators) is a fundamental problem in statistics.
At the same time, deviation-optimal model selection aggregation procedures have been used to
construct computable procedures (not necessarily computationally efficient)
to demonstrate the achievability of some statistical minimax lower bounds (see, e.g.,
\citep*{rakhlin2017empirical, mendelson2019unrestricted, mourtada2021distribution}).

One challenge concerning the analysis of optimal model selection aggregation
estimators is that only \emph{improper procedures} (i.e., whose ranges
$\mathcal{F}$ are strictly larger than the reference class $\mathcal{G}$)
can obtain optimal performance (that is, improperness is \emph{necessary}).
Regarding in-expectation bounds, optimal
performance is achievable via exponential weights (or progressive mixture)
algorithms, with different proofs available in the literature;
see, e.g., the works by \citet*{catoni1997mixture, yang2000combining, vovk2001competitive,
juditsky2008learning}. However, none of the proofs for the in-expectation
optimality of exponential weights algorithm follow traditional strategies
based on empirical processes theory, such as those based on local Rademacher
complexities (see Section 3.2.2 in the work by \citet*{audibert2010pac}).
As it turns out, a successful application of such strategies would be
impossible because they would lead to optimal exponential-tail deviation bounds
which were shown not to hold by
\citet{audibert2008progressive}. \citet{audibert2008progressive} also proposed
a deviation-optimal method for
model selection aggregation, called \emph{star algorithm}.
One of the key takeaways from Audibert's analysis is that the excess risk
random variable $\mathcal{E}(\widehat{f}, \mathcal{G})$ can take \emph{negative
values} for improper estimators $\widehat{f}$.
It follows that, in general, in-expectation
guarantees for improper methods cannot be used to derive high-probability
bounds because Markov's inequality does not apply.
For example, \citet*[Theorems 1 and
2]{mourtada2021distribution} exhibit two different statistical estimators
for the problem of linear regression, both of which satisfy expectation-optimal excess risk
bounds obtainable via average stability arguments, and both of which incur
excess risk lower bounded by an absolute constant, with a constant probability.

The phenomenon concerning deviation-optimality of model selection aggregation estimators
has generated a lot of attention;
for example, see the works by \citet*{
  lecue2009aggregation,
  rigollet2012kullback,
  dai2012deviation,
  lecue2014optimal,
  wintenberger2017optimal,
  bellec2017optimal} for analysis of different model selection aggregation
procedures.
More broadly, the analysis of improper statistical estimators is receiving
increased attention, as such procedures were shown to be necessary for optimal
statistical performance in
logistic regression, see \citep*{hazan2014logistic, pmlr-v75-foster18a,
  mourtada2019improper}, and linear regression, see \citep*{
  vaskevicius2020suboptimality, mourtada2021distribution}.

\subsection{Paper Outline and Summary of Main Results}

In this paper, we obtain \emph{exponential-tail} excess risk upper bounds that hold
for a \emph{general class} of estimators satisfying a certain geometric
condition that we call the \emph{offset condition} (see
Definition~\ref{dfn:offset-condition}).
This geometric condition  can serve as a design principle for
statistical estimators that satisfy sharp excess risk guarantees with high
probability. In particular,
arguments based on convex geometry can be used to establish that such a
condition holds for a broad class of known estimators
(see the examples in Section~\ref{sec:examples}).
The class of estimators satisfying the geometric condition
includes improper learning settings that are not covered by the
classical theory of local Rademacher complexities.
In the classical setting of empirical risk minimization performed over
a convex class under boundedness assumptions, our complexity measure
yields results \emph{at least as sharp} as those obtainable by the classical
theory of local Rademacher complexities (this is made more precise in
Section~\ref{sec:bernstein-vs-offset}). The starting point of our analysis is the work of
\citet*{liang2015learning}, who were the first to provide an
\emph{in-expectation} analysis of the star aggregation algorithm based on
\emph{offset Rademacher complexity}, a modified notion of classical
localization that arises from the analysis of \emph{offset empirical processes}.

The main contribution of the current paper is obtaining results analogous to
the ones achievable via the classical local Rademacher complexity theory,
yet applicable under a different set of assumptions. In particular,
the main element of the classical theory is an \emph{estimator-independent}
Bernstein condition (see Section~\ref{sec:background-classical-localization}
for details) that ensures a linear relationship between the variance
and expectation of the excess loss class. In contrast, our results build on an
\emph{estimator-dependent} geometric condition, called the offset condition.
The theory developed in this paper shows that the offset condition is
sufficient to ensure sharp excess risk guarantees for improper
estimators. For example, as discussed in Section~\ref{sec:examples},
any estimator that satisfies the offset condition while outputting a sparse
combination of a given finite dictionary of functions attains deviation-optimal
excess risk rate for the problem of model selection aggregation, where
improperness is necessary for optimality.

The rest of the paper is organized as follows.
\begin{itemize}
  \item In Section~\ref{sec:notation}, we summarize the notation used in this
    paper.
  \item In Section~\ref{sec:background-local-complexity}, we provide background
    on local Rademacher complexity measures. Section~\ref{sec:background-classical-localization}
    contains a sketch of how the classical theory of localization, through its
    foundation built on Talagrand's concentration inequality, is applicable in
    regimes where the variance of the excess loss class is controlled by a
    linear function of its expectation (which results in the use of the Bernstein
    condition for Lipschitz losses).
    In Theorem~\ref{thm:classical-localization}, we formulate an excess risk
    bound guaranteed via the classical theory for empirical risk minimization
    algorithms under the Bernstein condition. This result serves as a
    benchmark for our paper, which we aim to match without invoking the Bernstein
    condition. We achieve this (to the extent quantified in
    Section~\ref{sec:main-results})
    by establishing a general machinery of localization via offset Rademacher
    complexities,
    the background on which is provided in Section~\ref{sec:background-offset-localization}.
  \item The main results are presented in Section~\ref{sec:main-results}.
    \begin{itemize}
      \item Section~\ref{sec:main-results-definitions} contains the definition
        of the geometric condition (called the offset condition) that serves as our
        replacement of the Bernstein condition and the definition of offset
        Rademacher complexity, which is slightly modified from the one
        appearing in prior work by \citet*{liang2015learning}.
        Specifically, we include additional negative terms,
        which play an important role in our concentration arguments and in proving that our
        notion of complexity is never worse than the classical notion of local
        Rademacher complexities (cf.\ Lemma~\ref{lemma:offset-localization-not-worse}).
      \item
        Section~\ref{sec:main-results-multiplier-proposition} contains a
        moment generating function bound for offset multiplier empirical
        processes (Proposition~\ref{prop:multiplier-concentration}), which
        is the main technical contribution of the present paper.
        This result serves as our replacement for Talagrand's
        concentration inequality, on which the classical theory of localization is
        built. The key feature of our concentration result is the fact that
        the variance of the supremum of offset multiplier processes is
        automatically controlled by a linear function of their expectations due
        to the presence of the negative quadratic terms inside the supremum.
        In contrast, the classical theory of localization needs to
        \emph{assume} that a certain variance-expectation relationship holds,
        as elaborated in Section~\ref{sec:background-classical-localization}.
        We prove Proposition~\ref{prop:multiplier-concentration} via an
        application of an exponential Efron-Stein inequality as discussed in
        greater detail in Section~\ref{sec:proof-of-multiplier-proposition}.
      \item In Section~\ref{sec:main-results-rademacher-bound}, we present our
        main theorem -- an exponential-tail excess risk bound stated in terms of
        the offset Rademacher complexity (cf.\
        Theorem~\ref{thm:rademacher-bound}).
        The key difference from the usual theory of localization is that the
        estimator-independent Bernstein condition appearing in
        Theorem~\ref{thm:classical-localization} is replaced via the
        estimator-dependent offset condition. We prove
        Theorem~\ref{thm:rademacher-bound} by bounding the Laplace transform of the
        offset empirical processes (arising through the geometric condition
        imposed on an estimator) in terms of the Laplace transform of a related
        offset multiplier empirical process.
        We then complete the proof via
        an application of Proposition~\ref{prop:multiplier-concentration},
        which provides tight control on the Laplace transform of the obtained
        offset multiplier process.
      \item
        Further connections between the classical theory and the theory
        developed in this paper are discussed in Section~\ref{sec:bernstein-vs-offset}.
        In Lemma~\ref{lemma:offset-localization-not-worse}, we show that the
        offset Rademacher complexity is at most as large as the
        classical local Rademacher complexity. Thus, the bounds obtained in
        our paper, when they apply, are at least as sharp as those obtainable via the classical
        theory (cf.\ Corollary~\ref{corollary:offset-to-classical}).
        Finally, we discuss the sense in which the Bernstein condition and the
        offset condition can be considered as dual to one another,
        when the roles of empirical and population quantities are interchanged
        (cf.\ Lemma~\ref{lemma:bernstein-offset-duality}).
    \end{itemize}
  \item Section~\ref{sec:examples} contains several applications of
    the theory developed in this paper.
    In Lemma~\ref{lemma:sparse-offset-complexity}, we bound the offset Rademacher
    complexity of sparse linear classes; in Corollary~\ref{corollary:sparsity},
    we show how this bound can be applied for non-linear classes via a
    change-of-basis argument.
    As a direct consequence, we show how our theory can
    yield deviation-optimal bounds for two different model selection
    aggregation procedures, both of which output a sparse combination of
    dictionary elements and satisfy the offset condition.
    Such applications are outside the scope of the classical theory of
    localization, due to the necessary improperness of optimal estimators, as
    discussed in the introduction.
    Finally, we discuss how the analysis of iterative regularization schemes
    fits within the theory developed in this paper.
  \item Sections~\ref{sec:proof-of-rademacher-bound},
    \ref{sec:proof-of-multiplier-proposition} and
    Appendix~\ref{appendix:deferred-proofs} contain the proofs.
\end{itemize}

\subsection{Notation}
\label{sec:notation}
We denote by $P$ the unknown distribution from which an i.i.d.\ sample $S_{n} = (X_{i},
Y_{i})_{i=1}^{n}$ is drawn, where $(X_{i}, Y_{i}) \in \mathcal{X} \times \mathcal{Y}$.
We denote the marginal distribution on $\mathcal{X}$ by $P_{X}$ and for the
sample $S_{n} = (X_{i}, Y_{i})_{i=1}^{n}$, let $S_{n}^{X} = (X_{i})_{i=1}^{n}$.
An estimator $\widehat{f}$ is a mapping between datasets and some class of
predictors $\mathcal{F}$, called the range of the estimator $\widehat{f}$.
The loss function is denoted by $\ell : \mathcal{Y} \times \mathcal{Y} \to [0,
\infty)$. For any function $f : \mathcal{X} \to \R$, denote $\ell_{f}(X, Y) =
\ell(f(X), Y)$. The population risk functional is defined by $R(f) = \E
\ell_{f}(X,Y)$, where the expectation is computed with respect to $(X,Y) \sim
P$ and $f$ is always assumed to be measurable.
We say that the loss function $\ell$ is $L$-Lipschitz in its first argument
if for any $y,y_{1},y_{2} \in \mathcal{Y}$ we have $|\ell(y_{1}, y) -
\ell(y_{2}, y)| \leq L|y_{1} - y_{2}|$. As a function of the sample
$S_{n}$, define the empirical risk functional $R_{n}$ by $R_{n}(f) = n^{-1}\sum_{i=1}^{n}
\ell_{f}(X_{i}, Y_{i})$. The function class $\mathcal{F}$ always denotes the range
of some estimator, while $\mathcal{G}$ denotes a set of reference functions.
We let $g^{\star} \in \operatorname{argmin}_{g \in \mathcal{G}} R(g)$, assuming
without loss of generality that such a function exists; otherwise $g^{\star}$
could be replaced by some function that is arbitrarily close to attaining $\inf_{g \in
\mathcal{G}} R(g)$. For any function class $\mathcal{H}$, denote its star-hull by
$\operatorname{star}(\mathcal{H}) = \{\lambda h : h \in \mathcal{H}, \lambda
\in [0, 1]\}$, where $(\lambda h)(x) = \lambda h(x)$. We say that a function
class $\mathcal{H}$ is star-shaped (around the origin) if
$\operatorname{star}(\mathcal{H}) = \mathcal{H}$.
For any $\mathcal{F}$ and $g$, the class $\mathcal{F} - g$ denotes
$\{f - g : f \in \mathcal{F}\}$. Finally, we denote by $a \lesssim b$ the
existence of some universal constant $c > 0$ such that $a \leq cb$.


\section{Background on Local Complexity Measures}
\label{sec:background-local-complexity}

This section provides background on local complexity measures.
In Section~\ref{sec:background-classical-localization}, we introduce the
classical notion of local Rademacher averages, developed in the series of
works by \citet*{koltchinskii2000rademacher, koltchinskii2001rademacher,
bartlett2002model, lugosi2004complexity, bartlett2005local, koltchinskii2006local},
among others. In particular, we explain why this theory is primarily
applicable in the proper learning setup, and explain how convexity assumptions
enter this theory through the so-called Bernstein condition.
The present paper aims to replace such assumptions and establish a
methodology that applies to improper and non-convex problems of interest,
such as model selection aggregation.
In Section~\ref{sec:background-offset-localization}, we discuss a more recent
approach of localization via offset Rademacher complexities,
introduced in the statistical context with the quadratic loss by
\citet*{liang2015learning}; see also \citep*{rakhlin2014online}.
The offset Rademacher complexity approach replaces the Bernstein condition with
an estimator-dependent offset condition, and thus paves the way to achieve the
goals set out in this paper -- obtaining sharp \emph{exponential-tail} excess risk
guarantees that hold for improper estimators.

\subsection{Local Rademacher Complexity}
\label{sec:background-classical-localization}
Let $\mathcal{F}$ be the range of some estimator $\widehat{f}$, $\mathcal{G}$ be
a reference class of functions, and let $g^{\star}$ denote any population risk
minimizer over the class $\mathcal{G}$, i.e.,
$g^{\star} \in \text{argmin}_{f \in \mathcal{G}} R(f)$.
The first step in the classical local Rademacher complexity analysis proceeds by noting
that
\begin{align*}
  \mathcal{E}(\widehat{f}, \mathcal{G})
  &= (R(\widehat{f}) - R(g^{\star})) - (R_{n}(\widehat{f}) - R_{n}(g^{\star}))
  + (R_{n}(\widehat{f}) - R_{n}(g^{\star}))
  \\
  &\leq
  \sup_{f \in \mathcal{F}}\left\{
    (R(f) - R(g^{\star})) - (R_{n}(f) - R_{n}(g^{\star}))
  \right\}
  + (R_{n}(\widehat{f}) - R_{n}(g^{\star}))
\end{align*}
The term $R_{n}(\widehat{f}) - R_{n}(g^{\star})$ is typically controlled by assuming that
it is at most $0$ almost surely. This is true, for example, if $\widehat{f}$ is
an empirical risk minimizer over $\mathcal{F}$ and $\mathcal{G} \subseteq
\mathcal{F}$.

The supremum term is controlled via Talagrand's concentration
inequality\footnote{
  We state a version with absolute constants.
Of independent interest, various extensions and refinements of Talagrand's
concentration bound are available in the literature;
we refer the interested reader to
\citep*{ledoux1997talagrand, massart2000constants, bousquet2002bennett,
  klein2005concentration, mendelson2010empirical, lederer2014new}.
}
for empirical processes
\citep*{talagrand1994sharper}, a functional Bernstein-type concentration inequality with
variance proxy
\begin{equation}
  \label{eq:variance-proxy-talagrand}
  \sigma^{2}(\mathcal{F}) = \sup_{f \in \mathcal{F}}
  \left\{
    \operatorname{Var}_{(X,Y) \sim
    P}\left[\ell_{f}(X,Y) - \ell_{g^{\star}}(X, Y)\right]
  \right\}.
\end{equation}
In particular, denoting
$
  Z = \sup_{f \in \mathcal{F}}\{ (R(f) - R(g^{\star})) - (R_{n}(f) -
  R_{n}(g^{\star}))\}
$
and letting $c>0$ be some universal constant, for any $\delta \in (0,1)$ with
probability at least $1-\delta$ we have
\begin{align}
  \label{eq:talagrands-concentration}
  Z
  \leq
  2\E Z
  + c\sqrt{\frac{\sigma^{2}(\mathcal{F})\log(1/\delta)}{n}}
  + c \frac{C_{\ell}\log(1/\delta)}{n},
\end{align}
where $C_{\ell}$ is a boundedness
constant such that the for any $f \in \mathcal{F}$ and
any $(X,Y) \in \mathcal{X}\times{Y}$ we have $\abs{\ell_{f}(X,Y) -
\ell_{g^{\star}}(X,Y)} \leq C_{\ell}$.

Let us now informally discuss how the above concentration bound leads to a
localization theory via Rademacher complexities. Let $\psi(f, g^{\star}) \geq 0$
be some measure of distance between the functions $f$ and $g^{\star}$
(for the sake of this high-level presentation,
we ignore the properties that $\psi$ needs to satisfy).
The idea  of localization is to replace
$\mathcal{F}$ in \eqref{eq:talagrands-concentration} by a localized subset
$\mathcal{F}(r) = \{f \in \mathcal{F} : \psi(f, g^{\star}) \leq r\}$ for some
radius $r > 0$. The theory of local Rademacher complexities then aims
to compute the smallest value of $r > 0$ such that the supremum of the empirical
process computed over the localized class $\mathcal{F}(r)$ yields an upper
bound on the excess risk of an estimator of interest (typically the empirical risk
minimization estimator).

To allow for an explicit control of the variance proxy
$\sigma^{2}(\mathcal{F}(r))$, it is further assumed that for any $f \in
\mathcal{F}$, we have $\operatorname{Var}(\ell_{f} - \ell_{g^{\star}}) \leq
\psi(f, g^{\star})$. There are two consequences of the above assumed relation
between the variance and the distance function. First, it holds that
$\sigma^{2}(\mathcal{F}(r)) \leq r$. Second, it is possible to obtain a uniform
Bernstein-type concentration bound on the excess risk over the full class
$\mathcal{F}$, such that for each $f \in \mathcal{F}$, the variance-proxy is
proportional to $\sqrt{\psi(f, g^{\star})/n}$. For more details and a precise
quantification of the above statements we refer to
\citep*[Theorem 14.20, Equation 14.51]{wainwright2019high}.

When the obtained uniform Bernstein-type concentration bound is applied to the
estimator $\widehat{f}$ of interest, we obtain an upper bound on its excess
risk in terms of the supremum over a localized class $\mathcal{F}(r)$ (for some
radius $r > 0$), and the ``slow rate'' variance term $\sqrt{\psi(\widehat{f},
g^{\star})/n}$. To compensate for this variance term and to obtain a ``fast rate''
excess risk bound, it is further assumed that for some constant $B > 0$ the
following inequality holds for any $f \in \mathcal{F}$:
$\psi(f, g^{\star}) \leq B \E[\ell_{f} - \ell_{g^{\star}}]$.
Since the left hand side of the above equation is a non-negative distance, the
right hand side also needs to be non-negative. This, in turn, constrains us to
the settings where $\mathcal{F}$, the range of the estimator of interest,
cannot be larger than the reference class $\mathcal{G}$, for otherwise there would exist a
data generating distribution $P$ and a function $f \in \mathcal{F}$ such that
$\E[\ell_{f} - \ell_{g^{\star}}] < 0$.

Summing up the above, the theory of local Rademacher
complexities is rooted in the following variance-expectation assumption
-- a widely used condition in the empirical processes analysis of M-estimators
(see, e.g., the works by \citet*{geer2000empirical, massart2000some}):
\begin{equation}
  \label{eq:variance-controlled-by-expectation}
  \operatorname{Var}(\ell_{f} - \ell_{g^{\star}})
  \leq
  \psi(f, g^{\star})
  \leq
  B\E[\ell_{f} - \ell_{g^{\star}}]
  \quad
  \text{for any }
  f \in \mathcal{F}.
\end{equation}

In applications in learning theory, a natural choice for the distance function
$\psi$ is a suitably rescaled squared
$L_{2}(P_X)$ norm. Indeed, if the loss function $\ell$ is $C_{b}$-Lipschitz
in its first argument, then $\operatorname{Var}(\ell_{f} - \ell_{g^{\star}})
\leq C_{b}^{2}\E(f(X) - g^{\star}(X))^{2}$.
Thus, the remaining question is what is the smallest allowed value
$r > 0$ such that Talagrand's concentration inequality
\eqref{eq:talagrands-concentration} applied to $\mathcal{F}(r)$ yields an upper
bound on the excess risk $\mathcal{E}(\widehat{f}, \mathcal{G})$. Using a
peeling argument applied to a reweighted excess loss class (cf.
\citet*[Section 3]{bartlett2005local}), this value
can be shown to equal a solution to a certain fixed-point equation, leading to
the following definition.

\begin{definition}[Local Rademacher Complexity]
  \label{dfn:local-complexity}
  Let $P_{X}$ denote any distribution supported on $\mathcal{X}$
  and let $\mathcal{H}$ denote any class of functions mapping $\mathcal{X}$ to
  $\mathbb{R}$. For $r > 0$, let
  $\mathcal{H}(r) = \{h \in \mathcal{H} : \E_{X \sim P_{X}} [h(X)^{2}] \leq r\}$.
  Let $\sigma = (\sigma_{i})_{i=1}^{n}$ be a
  sequence of i.i.d.\ Rademacher (i.e., symmetric and
  $\{\pm1\}$-valued) random variables and let $S_{n}^{X} = (X_{i})_{i=1}^{n}$
  denote $n$ independent random variables distributed according to $P_{X}$.
  Then, for any $\gamma > 0$, the local Rademacher complexity
  of the class $\mathcal{H}$ is defined by
  $$
    \mathfrak{R}^{\text{loc}}_{n}(P_{X}, \mathcal{H}, \gamma)
    = \inf\left\{
      r > 0
      :
      \mathbf{E}_{S_{n}^{X}, \sigma}
      \left[
        \sup_{h \in \mathcal{H}(\gamma^{-1} r)} \left\{
          \frac{1}{n}\sum_{i=1}^{n}\sigma_{i}h(X_{i})
        \right\}
      \right]
      \leq r
    \right\}.
  $$
\end{definition}

It now remains to discuss when the second inequality of
\eqref{eq:variance-controlled-by-expectation} holds in a
\emph{distribution-free}\footnote{
  Recall that, as discussed in the introduction, the present paper aims to
  obtain excess risk bounds that hold for any distribution $P$ supported on
  $\mathcal{X} \times \mathcal{Y}$.
}
sense (as opposed to, e.g., imposing low-noise
assumptions on the underlying distribution, as is frequently done in the
classification setting). The primary application domain
where this is true is when a function class $\mathcal{F}$ is convex,
$g^{\star} \in \mathcal{G}$ denotes a population risk
minimizer over all functions in $\mathcal{F}$ (thus, $\mathcal{F} \subseteq
\mathcal{G}$, constraining to study the proper learning setting),
and the loss function $\ell$ is strongly convex
in its first argument (cf. \citet*[Section 5.2]{bartlett2005local}).
The second inequality in
\eqref{eq:variance-controlled-by-expectation}, when $\psi$ is taken to be the
squared $L_{2}(P_{X})$ norm, is often called the \emph{Bernstein
condition} (cf.\ \citet*{bartlett2006empirical}), which we state below.

\begin{definition}[Bernstein Condition]
  \label{dfn:bernstein-condition}
  Let $P$ be a distribution supported on $\mathcal{X} \times
  \mathcal{Y}$ and let $\ell$ be a loss function with domain
  $\mathcal{Y} \times \mathcal{Y}$.
  The tuple $(P, \ell, \mathcal{F}, g^{\star})$ satisfies the Bernstein
  condition with parameter $\gamma > 0$ if the following holds for any
  $f \in \mathcal{F}$:
  \begin{equation}
    \label{eq:bernstein-condition}
    \E_{X \sim P_{X}}(f(X) - g^{\star}(X))^{2}
    \leq
    \frac{1}{\gamma} \E_{(X,Y) \sim P} \left[ \ell_{f}(X, Y) -
    \ell_{g^{\star}}(X, Y)\right].
  \end{equation}
\end{definition}

Summing up all of the above, let us now state a result obtained by
\citet*{bartlett2005local}. In our notation, it reads as follows.

\begin{theorem}[Corollary 5.3 in \citep{bartlett2005local}]
  \label{thm:classical-localization}
  Let $\mathcal{F}$ be a class of functions mapping $\mathcal{X}$ to
  $[-b,b]$ for some $b > 0$. Let $P$ be a distribution supported on
  $\mathcal{X} \times [-b, b]$ and let $g^{\star} \in \text{argmin}_{g \in
  \mathcal{G}} R(g)$, where $\mathcal{G}$ is some reference class of functions.
  Suppose that the following three conditions hold:
  \begin{enumerate}
    \item The loss function $\ell : [-b,b] \times [-b,b] \to [0, \infty)$ is
      $C_{b}$-Lipschitz in its first argument;
    \item The tuple $(P, \ell, \mathcal{F}, g^{\star})$ satisfies the Bernstein
      condition with parameter $\gamma > 0$;
    \item The function class $\mathcal{F} - g^{\star} = \{f - g^{\star} : f \in
      \mathcal{F}\}$ is star-shaped around $0$ (cf.\
      Section~\ref{sec:notation}).
  \end{enumerate}
  Let $\widehat{f}$ be an estimator such that $R_{n}(\widehat{f}) -
  R_{n}(g^{\star}) \leq 0$ almost surely. Then, for any $\delta \in (0,1)$
  with probability at least $1-\delta$, we have
  \begin{equation}
    \mathcal{E}(\widehat{f}, \mathcal{G})
    \leq c_{1} C_{b} \mathfrak{R}^{\text{loc}}_{n}(P_{X}, \mathcal{F} - g^{\star},
    C_{b}^{-1}\gamma)
    + c_{2}\frac{(C_{b}b + C_{b}^{2}\gamma^{-1})\log(1/\delta)}{n},
  \end{equation}
  where $c_{1},c_{2} > 0$ are universal constants.
\end{theorem}

\paragraph{Limitations.}
We conclude this section by briefly summarizing two limitations of the above framework.

The first limitation is its reliance on the Bernstein
condition. As already discussed, a natural application domain
where this condition holds, together with the condition that $R_{n}(\widehat{f}) -
R_{n}(g^{\star}) \leq 0$ almost surely, is when $\mathcal{F} = \mathcal{G}$
and $\mathcal{F}$ is a convex class.
Since improper learning settings do not satisfy the Bernstein condition
uniformly for all data generating distributions $P$,
Theorem~\ref{thm:classical-localization} does not easily lend itself to
non-convex and improper application domains that arise,
for instance, in model selection aggregation or iterative regularization applications (cf.\
Section~\ref{sec:examples}). The present paper addresses these limitations
(see, in particular, Theorem~\ref{thm:rademacher-bound} and example
applications in Section~\ref{sec:examples}).

The second limitation concerns the boundedness assumptions, also
present in our work. Such
assumptions prevent us from capturing unbounded, and in particular,
heavy-tailed problems that have recently received a lot of attention; see the
survey by \citet*{lugosi2019mean}. For progress in this direction, we refer to
the works by \citet*{mendelson2015learning, mendelson2018learning,
oliveira2016lower}, where one-sided concentration
arguments and moment-equivalence assumptions play a central role.
The above line of work provides powerful tools for treating many unbounded and
potentially heavy-tailed problems of interest, that fall outside of the scope
of the present paper.
However, let us remark that such assumptions do not allow for immediate
distribution-free treatment of the bounded setting considered in our work;
see, for example, the discussions in \citep*{saumard2018optimality,
vaskevicius2020suboptimality}.

\subsection{Offset Rademacher Complexity}
\label{sec:background-offset-localization}

We now describe the offset Rademacher complexity approach due to
\citet*{liang2015learning}, an empirical processes theory-based technique shown to yield
distribution-free \emph{in-expectation} guarantees for Audibert's star
algorithm in the bounded setting.
Let us preface the rest of this section by
noting that the analysis in the above-cited paper is constrained to the case
when $\ell$ is the quadratic loss, i.e., for any $y, y'$ we have $\ell(y,y') =
(y - y')^{2}$.

Let $\mathcal{G} = \{g_{1}, \dots, g_{m}\}$ denote a dictionary of $m$
functions mapping $\mathcal{X} \to [-b, b]$. Then, as discussed in the introduction,
any estimator whose range $\mathcal{F}$ is equal to $\mathcal{G}$ (i.e., any proper
estimator) can only yield slow excess risk
rates of order $n^{-1/2}$ instead of the optimal rate $b^{2}\log(m)/n$.
Hence, due to the necessary improperness of optimal estimators, the model
selection aggregation problem does not easily fit into the classical theory of
localization discussed in the previous section.
The optimal in-expectation and in-deviation performance is attained by the star
estimator $\widehat{f}^{\text{(star)}}$ due to \citet*{audibert2008progressive},
defined as follows:
$$
  \widehat{f}^{\text{(star)}}
  = \operatorname{argmin}_{f \in \mathcal{G}, \lambda \in [0, 1]}
  R_{n}(\lambda \widehat{f}^\text{(ERM)} + (1-\lambda)f),
  \,\,\text{where}\,\,
  \widehat{f}^{\text{(ERM)}}
  = \operatorname{argmin}_{f \in \mathcal{G}}
  R_{n}(f).
$$
Recall that in the above expressions $R_{n}$ denotes the empirical risk functional.

The key observation of \citet*[Lemma 1]{liang2015learning} is that the star
estimator satisfies a \emph{deterministic} condition that we state below.
For any observed sample $S_{n} = (X_{i},
Y_{i})_{i=1}^{n}$, the following holds with a constant $\gamma = 1/18$:
\begin{equation}
  \label{eq:star-quadratic-loss-offset}
  R_{n}(\widehat{f}^{\text{(star)}}) - R_{n}(g^{\star})
  \leq -\gamma
  \sum_{i=1}^{n} (f^{\text{(star)}}(X_{i}) - g^{\star}(X_{i}))^{2}.
\end{equation}
The above condition can be interpreted as a dual to the Bernstein condition
(cf.\ Definition~\ref{dfn:bernstein-condition}), with population quantities
replaced by its empirical counterparts (see Section~\ref{sec:bernstein-vs-offset} and
Lemma~\ref{lemma:bernstein-offset-duality}); however, the above inequality does
not require the estimator $f^{\text{(star)}}$ to be proper (in fact, it is
improper), nor does it require its range $\mathcal{F}$ to be convex.
We defer an extended discussion to Section~\ref{sec:bernstein-vs-offset}.

The condition \eqref{eq:star-quadratic-loss-offset} can be used to upper bound
the excess risk as follows:
\begin{align*}
  &\mathcal{E}(f^{\text{(star)}}, \mathcal{G})  \\
  &= (R(f^{\text{(star)}}) - R(g^{\star}))
  - (R_{n}(f^{\text{(star)}}) - R_{n}(g^{\star}))
  + (R_{n}(f^{\text{(star)}}) - R_{n}(g^{\star})) \\
  &\leq
  (R(f^{\text{(star)}}) - R(g^{\star}))
  - (R_{n}(f^{\text{(star)}}) - R_{n}(g^{\star}))
  - \gamma \frac{1}{n}\sum_{i=1}^{n}
  (f^{\text{(star)}}(X_{i}) - g^{\star}(X_{i}))^{2} \\
  &\leq
  \sup_{f \in \mathcal{F}}
  \big\{
    (R(f) - R(g^{\star}))
    - (R_{n}(f) - R_{n}(g^{\star}))
    - \gamma \frac{1}{n}\sum_{i=1}^{n}
    (f(X_{i}) - g^{\star}(X_{i}))^{2}
  \big\}.
\end{align*}
Taking expectations on both sides and applying classical symmetrization and
contraction arguments, \citet*[Theorem 3]{liang2015learning} show that
the following holds for some absolute constants $c_{1},c_{2}>0$:
\begin{equation}
  \label{eq:offset-expected-bound}
  \E_{S_{n}}
  \mathcal{E}(f^{\text{(star)}}, \mathcal{G})
  \leq c_{1}b \E_{S_{n}, \sigma}\bigg[
  \sup_{h \in \mathcal{F} - g^{\star}}
  \bigg\{
    \frac{1}{n}\sum_{i=1}^{n} \sigma_{i}h(X_{i})
    - \frac{\gamma}{b}h(X_{i})^{2}
  \bigg\}
  \bigg],
\end{equation}
where $\sigma = (\sigma_{1},\dots,\sigma_{n})$ denotes a sequence of i.i.d.\
Rademacher random variables. The right-hand side of the above equation is called the \emph{offset
Rademacher complexity} of the class $\mathcal{F} - g^{\star}$;
the negative quadratic terms produce a localization
phenomenon similar to that of Definition~\ref{dfn:local-complexity}.
As we shall see in Corollary~\ref{corollary:offset-to-classical}, a modified
notion of the above complexity measure yields guarantees at least as
sharp as those obtainable via local Rademacher complexities introduced in the
previous section.

\paragraph{Limitations.} We now discuss the limitations of the existing results
based on the above approach. First, the bound \eqref{eq:offset-expected-bound}
holds only \emph{in-expectation}. However, the star estimator was introduced
to address the \emph{in-deviation} optimality for model selection aggregation,
and thus, obtaining in-deviation guarantees for this estimator are of
particular interest \citep{audibert2008progressive}.
As discussed in the introduction, transforming in-expectation
guarantees to in-deviation guarantees for \emph{improper} statistical
estimators presents several technical difficulties. High probability
alternatives to the bound \eqref{eq:offset-expected-bound} have not been
obtained before our work since there is no replacement for Talagrand's
concentration inequality on which the classical theory of localization resides.
We develop such a (one-sided) concentration result in
Proposition~\ref{prop:multiplier-concentration}, using which we obtain an
exponential-tail offset Rademacher complexity deviation bound in
Theorem~\ref{thm:rademacher-bound}.

While high probability bounds in terms of offset Rademacher complexity have not
been previously developed, let us now discuss some deviation bounds that have
been obtained using the framework described above.
The primary high probability result obtained in \citep*[Theorem 4]{liang2015learning}
is not distribution-free because it relies on an additional lower isometry
assumption. In addition, it upper bounds the excess risk in terms of
another \emph{random variable} no easier to control than the excess risk itself;
further control on this random variable is only shown for finite classes or
their star-hulls. The obtained bound for star-hulls of finite classes is then
used to obtain a high-probability excess risk bound for Audibert's star
algorithm.
However, due to the use of covering-number arguments, the resulting excess risk bounds
suffer from excess logarithmic terms; see \citep*[Lemma 11]{liang2015learning}.

The very recent work of \citet*{vijaykumar2021localization} extends the
geometric inequality \eqref{eq:star-quadratic-loss-offset} to general loss
functions. However, the high probability bounds obtained therein
are expressed in terms of empirical covering numbers where the covering is
performed with the \emph{worst-case} metric. In contrast, local Rademacher
complexity (cf.\ Definition~\ref{dfn:local-complexity}) can be upper
bounded using covering number arguments where the covering is performed with
the $L_{2}(P_{X})$ metric, leading to minimax optimal bounds in many cases (see
\citet*[Chapters 13 and 14]{wainwright2019high} for some examples).
Crucially, in general the notion of complexity based on empirical covering numbers using
worst-case metric used by \citet*{vijaykumar2021localization}
does not capture statistical minimax
optimality and results in suboptimal bounds even for the star
estimator applied to a problem with a \emph{finite} reference class $\mathcal{G}$.
In contrast, we show in Section~\ref{sec:examples} how the geometric
inequality obtained by \citet*{vijaykumar2021localization}, when used
with offset Rademacher complexity bounds developed in this paper, results in
minimax optimal bounds for the star aggregation algorithm.


\section{Main Results}
\label{sec:main-results}

The main results of this paper are presented in this section. In
Section~\ref{sec:main-results-definitions},
we introduce the geometric condition (called the offset condition) used to replace
the Bernstein condition; further, we define the offset Rademacher complexity
(slightly modified from the one appearing in prior works)
used to replace the classical notion of local Rademacher complexity.
Section~\ref{sec:main-results-multiplier-proposition}
contains a moment generating function bound for shifted multiplier empirical processes.
This result serves as our replacement for Talagrand's concentration inequality,
the foundation of the classical theory of localization.
Section~\ref{sec:main-results-rademacher-bound} contains a high probability
excess risk bound in terms of the offset Rademacher complexity; this result
applies in settings where the Bernstein condition does not hold.  Finally, in
Section~\ref{sec:bernstein-vs-offset}, we provide a comparison between the offset
and Bernstein conditions and demonstrate that the theory presented in this paper can
recover the classical distribution-free bounds overviewed in
Section~\ref{sec:background-classical-localization}.

\subsection{Definitions}
\label{sec:main-results-definitions}

We begin with the definition of the \emph{offset condition}. Observe that this
condition is \emph{estimator-dependent}, as opposed to the Bernstein condition
(cf.\ Definition~\ref{dfn:bernstein-condition}).

\begin{definition}[Offset Condition]
  \label{dfn:offset-condition}
  Let $\mathcal{G}$ be a class of functions mapping $\mathcal{X}$ to $[-b, b]$
  for some $b > 0$. Fix a loss function
  $\ell : [-b, b] \times [-b, b] \to [0, \infty)$ and recall that $R_{n}$ denotes
  the induced empirical risk functional.
  Let $\varepsilon : [0,1] \to \R$ be some function and let $\gamma > 0$ be
  some positive real number.
  Let $P$ be a distribution supported on $\mathcal{X} \times \mathcal{Y}$.
  An estimator $\widehat{f}$ satisfies the \emph{offset condition}
  with respect to $(\mathcal{G}, \ell, \varepsilon, \gamma)$
  for the distribution $P$, if for any
  any $\delta \in [0,1]$ the following holds:
  $$
    \P_{S_{n}}\bigg(
      R_{n}(\widehat{f}) - R_{n}(g^{\star})
      \leq -\gamma \sum_{i=1}^{n}(\widehat{f}(X_{i}) - g^{\star}(X_{i}))^{2}
      + \varepsilon(\delta)
    \bigg)
    \geq 1 - \delta,
  $$
  where $S_{n} = (X_{i}, Y_{i})_{i=1}^{n}$ is an i.i.d.\ sample drawn
  from the distribution $P$
  and $g^{\star} = g^{\star}(\mathcal{G}, P, \ell)$ denotes any population
  risk minimizer in the class $\mathcal{G}$.

  Whenever the following deterministic inequality holds
  for any sample $S_{n} = (X_{i}, Y_{i})_{i=1}^{n} \in (\mathcal{X} \times
  \mathcal{Y})^{n}$:
  $$
    R_{n}(\widehat{f}) - R_{n}(g^{\star})
      \leq -\gamma \sum_{i=1}^{n}(\widehat{f}(X_{i}) - g^{\star}(X_{i}))^{2}
      + \varepsilon,
  $$
  we say that the estimator $\widehat{f} = \widehat{f}(S_{n})$
  satisfies the \emph{deterministic offset condition} with
  respect to $(\mathcal{G}, \ell, \varepsilon, \gamma)$.
\end{definition}
In the above definition the function $\varepsilon(\cdot)$ allows for the
offset condition to fail with probability $\delta$, while incurring a penalty
$\varepsilon(\delta)$. As we shall see in Section~\ref{sec:examples}, such a
condition naturally enters the analysis of some improper estimators.
Also, we will discuss some example estimators that satisfy the deterministic
offset condition.

In Section~\ref{sec:background-classical-localization}, we described how
the Bernstein condition implies local Rademacher complexity excess risk bounds for
empirical risk minimization estimators. Likewise, we shall see that offset
condition implies excess risk bounds expressed in terms of the offset
Rademacher complexity defined below.

\begin{definition}[Offset Rademacher Complexity]
  \label{dfn:offset-complexity}
  Let $P_{X}$ be any distribution supported on $\mathcal{X}$
  and let $\mathcal{H}$ be any class of functions mapping $\mathcal{X}$ to
  $\mathbb{R}$.
  Let $\sigma = (\sigma_{i})_{i=1}^{n}$ denote a
  sequence of i.i.d.\ Rademacher (i.e., symmetric and
  $\{\pm1\}$-valued) random variables and let $S_{n}^{X} = (X_{i})_{i=1}^{n}$
  denote $n$ independent random variables distributed according to $P_{X}$.
  Then, for any $\gamma > 0$, the offset Rademacher complexity
  of the class $\mathcal{H}$ is defined by
  $$
    \mathfrak{R}_{n}^{\text{off}}(P_{X}, \mathcal{H}, \gamma) =
    \E_{S_{n}^{X}, \sigma}\left[
      \sup_{h \in \mathcal{H}}\left\{
        \frac{1}{n}
        \sum_{i=1}^{n} \sigma_{i}h(X_{i}) - \gamma h(X_{i})^{2}
        - \gamma \E_{X \sim P_{X}}[h(X)^{2}]
      \right\}
    \right].
  $$
\end{definition}

Let us remark that our definition above differs from the one presented in
Section~\ref{sec:background-offset-localization} since we include extra negative
terms $-\gamma \E_{X \sim P_{X}}[h(X)^{2}]$ inside the above supremum.
This refinement is necessary for our concentration argument to work, since we
establish moment bounds for shifted multiplier processes that contain negative
population terms (cf.\ Section~\ref{sec:main-results-multiplier-proposition}).
At the same time, the inclusion of the negative quadratic population terms allows us to
show
that the above notion of
complexity is at least as sharp as the classical one introduced in
Definition~\ref{dfn:local-complexity}
(see Lemma~\ref{lemma:offset-localization-not-worse} in
Section~\ref{sec:bernstein-vs-offset} for details).

\subsection{Concentration of Shifted Multiplier Processes}
\label{sec:main-results-multiplier-proposition}

The primary technical tool in this paper is the following proposition, which
proves a Bernstein-type one-sided concentration bound for the supremum of
shifted multiplier processes (defined below in
Equation~\eqref{eq:dfn-multiplier}). This proposition plays a crucial role in
establishing our main result, Theorem~\ref{thm:rademacher-bound} presented in
the next section. In particular, provided that an estimator
satisfies the offset condition, we will show that the moment generating
function of its excess risk can be controlled by the moment generating function
of a certain shifted multiplier process.
We defer the proof of the below proposition to
Section~\ref{sec:proof-of-multiplier-proposition}.

\begin{proposition}
  \label{prop:multiplier-concentration}
  Let $\mathcal{H}$ be a class of functions mapping $\mathcal{X}$ to
  $\R$. Let $P_{(X, \zeta)}$ be a joint distribution on
  $\mathcal{X} \times \R$ with marginal distributions $P_{X}$ and $P_{\zeta}$,
  and let $S_{n} = (X_{i}, \zeta_{i})_{i=1}^{n}$
  be a set of $n$ i.i.d.\ samples from $P_{(X,\zeta)}$. Fix any positive
  constant $\gamma > 0$ and define a random variable $U = U(S_{n})$ to be the
  supremum of the offset multiplier process as follows:
  \begin{equation}
    \label{eq:dfn-multiplier}
    U =
    \sup_{h \in \operatorname{star}(\mathcal{H})}\left\{
      \sum_{i=1}^{n} \zeta_{i}h(X_{i}) -
      \E_{(X, \zeta) \sim P_{(X, \zeta)}}[\zeta h(X)]
      - \gamma h(X_{i})^{2}
      - \gamma \E_{X \sim P_{X}}[h(X)^{2}]
    \right\}.
  \end{equation}
  Suppose that there exist positive constants $\kappa$ and $\sigma$ such that
  $
    \sup_{h \in \mathcal{H}} \norm{h}_{L_{\infty}(P_{X})} \leq \kappa
  $
  and
  $
    \norm{\zeta}_{L_{\infty}(P_{\zeta})} \leq \sigma.
  $
  Then, for $\eta = 8(\sigma^{2}\gamma^{-1} + \gamma\kappa^{2})$ and any
  $\lambda \in (0, 1/\eta)$ the following holds:
  \begin{equation}
    \label{eq:multiplier-process-mgf-bound}
    \log \E e^{\lambda(U - \E U)} \leq
    \frac{\lambda^{2}\eta\E U}{2(1 - \eta\lambda)}.
  \end{equation}
\end{proposition}
Before turning to the offset Rademacher complexity upper bounds, let us remark that
in the above moment bound, the variance proxy/variance factor (in the sense of
\citep[Section 2.4]{boucheron2013concentration}) is equal to $\eta \E U$; thus
the variance of the random variable $U$ is automatically controlled by its expectation.
In particular, the above bound can be transformed into deviation bounds of the form $U \leq 2\E[U] +
c \eta \log(1/\delta)$, where $\delta > 0$ is the confidence parameter. In
contrast, recall that the variance proxy in Talagrand's concentration
inequality \eqref{eq:talagrands-concentration} is not controlled by the
expectation of the corresponding empirical process, which in turn leads to the
localization machinery where Rademacher averages need to be computed over
explicitly constrained subsets of the function class of interest, and where the
Bernstein condition is imposed to compensate for the resulting variance terms.
On the other hand, using the above concentration result, our theory allows us to obtain
high probability bounds in terms of the offset Rademacher complexity without
relying on the Bernstein condition, as we
show in the following section.

\subsection{Exponential-Tail Offset Rademacher Complexity Bound}
\label{sec:main-results-rademacher-bound}

We now present the main result of this paper, the proof of which can be found
in Section~\ref{sec:proof-of-rademacher-bound}. The following theorem provides
an alternative to Theorem~\ref{thm:classical-localization}, but with Bernstein
condition replaced via the offset condition. As a consequence, the offset condition can
serve as a design principle for estimators in the regimes where the Bernstein
condition fails to hold; some examples are given in Section~\ref{sec:examples}.

\begin{theorem}
  \label{thm:rademacher-bound}
  Let $\widehat{f}$ be an estimator with range $\mathcal{F}$,
  where $\mathcal{F}$ denotes a class of functions mapping $\mathcal{X}$ to
  $[-b,b]$ for some $b > 0$. Let $P$ be any distribution supported on
  $\mathcal{X} \times [-b, b]$ and denote $g^{\star} \in \text{argmin}_{g \in
  \mathcal{G}} R(g)$, where $\mathcal{G}$ is some reference class of functions.
  Suppose that the following two conditions hold:
  \begin{enumerate}
    \item The loss function $\ell : [-b,b] \times [-b,b] \to [0, \infty)$ is
      $C_{b}$-Lipschitz in its first argument;
    \item The estimator $\widehat{f}$ satisfies the offset condition with
      respect to $(\mathcal{G}, \ell, \varepsilon, \gamma)$ for the
      distribution $P$, where $\varepsilon$ is some function mapping
      $[0,1]$ to $\R$ and $\gamma > 0$ is some positive real number.
  \end{enumerate}
  Then, for any $\delta_{1}, \delta_{2} \in (0,1)$ with probability at least
  $1- \delta_{1} - \delta_{2}$, we have
  \begin{equation}
    \mathcal{E}(\widehat{f}, \mathcal{G})
    \leq c_{1} C_{b}' \mathfrak{R}_{n}^{\text{off}}(P_{X}, \operatorname{star}(\mathcal{F} -
    g^{\star}),
    (C_{b}')^{-1}\gamma)
    + c_{2}\frac{\gamma^{-1}(C'_{b})^{2}\log(1/\delta_{1})}{n}
    + \varepsilon(\delta_{2}),
  \end{equation}
  where $c_{1},c_{2} > 0$ are some universal constants and $C_{b}' = C_{b} +
  \gamma b$.
\end{theorem}

\begin{remark}
  \label{remark:lipschitz-parameter-size}
  In comparison with Theorem~\ref{thm:classical-localization}, the
  above result replaces $C_{b}$ with a worse constant $C_{b}' = C_{b} + \gamma
  b$. However, the primary application domain where the above theorems hold is
  the setting where for any $y \in [-b, b]$, the function
  $\ell(\cdot, y)$ is $C_{b}$-Lipschitz and $\gamma$-strongly convex
  in the fist argument (see Section~\ref{sec:examples} for examples).
  In such a setting it can be shown that $\gamma b \leq C_{b}$ and hence
  $C_{b}' \leq 2C_{b}$.
\end{remark}

\subsection{Recovering Local Rademacher Complexity Results Without The Bernstein Condition}
\label{sec:bernstein-vs-offset}

In this section, we discuss how Theorem~\ref{thm:rademacher-bound} yields
excess risk bounds that are no worse than the ones stated in
Theorem~\ref{thm:classical-localization}. We begin by stating the following
lemma, which is proved in Appendix~\ref{sec:proof-of-lemma-offset-not-worse}.

\begin{lemma}
  \label{lemma:offset-localization-not-worse}
  Let $P_{X}$ be any distribution supported on $\mathcal{X}$
  and let $\mathcal{H}$ be any star-shaped class of functions (i.e.,
  $\mathcal{H} = \operatorname{star}(\mathcal{H})$) mapping
  $\mathcal{X}$ to $\mathbb{R}$. Then, for any $\gamma > 0$ we have
  $$
    \mathfrak{R}_{n}^{\text{off}}(P_{X}, \mathcal{H}, \gamma)
    \leq \mathfrak{R}_{n}^{\text{loc}}(P_{X}, \mathcal{H}, \gamma).
  $$
\end{lemma}

An immediate consequence of the above lemma is the following corollary, which
shows that the classical local Rademacher complexity bounds hold when the
Bernstein condition is replaced via the \emph{estimator-dependent} offset condition.

\begin{corollary}
  \label{corollary:offset-to-classical}
  Consider the setting of Theorem~\ref{thm:rademacher-bound}.
  For any $\delta_{1}, \delta_{2} \in (0,1)$ with probability at least
  $1-\delta_{1}-\delta_{2}$, we have
  \begin{equation}
    \mathcal{E}(\widehat{f}, \mathcal{G})
    \leq c_{1} C_{b}' \mathfrak{R}_{n}^{\text{loc}}(P_{X}, \operatorname{star}(\mathcal{F} -
    g^{\star}),
    (C_{b}')^{-1}\gamma)
    + c_{2}\frac{\gamma^{-1}(C'_{b})^{2}\log(1/\delta_{1})}{n}
    + \varepsilon(\delta_{2}),
  \end{equation}
  where $c_{1},c_{2} > 0$ are some universal constants and $C_{b}' = C_{b} +
  \gamma b$.
\end{corollary}

It remains to discuss the relationship between the offset and Bernstein
conditions. A typical example where the Bernstein condition holds for
\emph{any} distribution $P$ is when $\mathcal{F} = \mathcal{G}$ is a convex class, and
the loss function is strongly convex. In such regimes, any empirical risk
minimizer over $\mathcal{F}$ satisfies the offset condition. Thus,
when applied to an empirical risk minimization estimator, the offset condition can be seen as
a dual condition to the Bernstein condition, where the roles played by
empirical and population quantities are interchanged.
We formalize this observation in the lemma below.

\begin{lemma}
  \label{lemma:bernstein-offset-duality}
  Let $\mathcal{F}$ be a class of functions mapping $\mathcal{X}$ to $\R$.
  Let $\ell : \mathcal{Y} \times \mathcal{Y} \to [0, \infty)$ be a loss
  function and let $\mathcal{P}_{\mathcal{X} \times \mathcal{Y}}$ be the
  set of all distributions $P$ supported on $\mathcal{X} \times \mathcal{Y}$.
  Let $f^{\star} = f^{\star}(\mathcal{F}, P, \ell)$ be any population risk minimizer
  over $\mathcal{F}$. Let $\widehat{f}^{\text{(ERM)}}$ be an estimator that
  returns any empirical risk minimizer in the class $\mathcal{F}$.
  If for any $P \in \mathcal{P}_{\mathcal{X} \times \mathcal{Y}}$
  the tuple $(P, \ell, \mathcal{F}, f^{\star})$ satisfies
  the Bernstein condition with parameter $\gamma$, then the estimator
  $\widehat{f}^{\text{(ERM)}}$ satisfies the deterministic offset condition with
  respect to $(\mathcal{F}, \ell, 0, \gamma)$.
\end{lemma}

\begin{proof}
  Given an i.i.d.\ sample $S_{n} = (X_{i}, Y_{i})_{i=1}^{n}$ from some
  distribution $P \in \mathcal{P}_{\mathcal{X} \times \mathcal{Y}}$,
  let $P_{n}$ denote a distribution on $\mathcal{X} \times \mathcal{Y}$
  assigning equal mass to each $(X_{i}, Y_{i})$. Since $P_{n} \in
  \mathcal{P}_{\mathcal{X} \times \mathcal{Y}}$, by the assumption of this
  lemma $(P_{n}, \ell, \mathcal{F}, \widehat{f}^{(ERM)}(S_{n}))$ satisfies the
  Bernstein condition with parameter $\gamma$. This is equivalent to saying
  that $\widehat{f}^{\text{(ERM)}}$ satisfies the deterministic offset
  condition with respect to $(\mathcal{F}, \ell, 0, \gamma)$.
\end{proof}

Let us conclude this section by highlighting one difference between the offset
and Bernstein conditions.
In some settings, the Berstein condition is used as a \emph{distributional}
assumption, which imposes constraints on the data distribution itself --
as opposed to \emph{distribution-free} results, holding for any distribution.
For example, in the classification setting
with zero-one loss, the Bernstein condition corresponds to bounded noise assumptions (see the
discussions in \citep*{boucheron2005theory}), under which empirical risk
minimization estimator can achieve fast rates of convergence of the excess risk.
For sharp treatment of the classification setting under the bounded noise
assumptions via ideas related to offset Rademacher averages, see
\citep*{zhivotovskiy2018localization}. At the same time, let us remark that
the offset condition can be exploited to design statistical estimators
that achieve fast rates in the classification setting in a distribution-free
sense (i.e., without bounded noise assumptions),
provided an option to abstain from prediction exists; for an extended discussion see
\citep*{bousquetfast}.


\section{Examples}
\label{sec:examples}

In this section, we discuss some applications of our theory to problems where
the Bernstein condition does not hold, yet there exist estimators
that satisfy the offset condition. As a result, sharp deviation-optimal
excess risk rates can be obtained for such estimators via the theory developed
in this paper.

For any function class $\mathcal{H}$ mapping $\mathcal{X}$ to $\R$ and
any sample $S_{n}^{X} = (X_{i})_{i=1}^{n}$, where $X_{i} \in \mathcal{X}$,
define
$$
  \mathfrak{R}^{\text{off}}(S_{n}^{X}, \mathcal{H}, \gamma)
  = \E_{\sigma}\left[
      \sup_{h \in \mathcal{H}}\left\{
        \frac{1}{n}
        \sum_{i=1}^{n} \sigma_{i}h(X_{i}) - \gamma h(X_{i})^{2}
      \right\}
      \bigg\vert S_{n}^{X}
    \right],
$$
where $\sigma = (\sigma_{1}, \dots, \sigma_{n})$ denotes a sequence of i.i.d.\
Rademacher random variables. Observe, in particular, that for any distribution
$P_{X}$ supported on $\mathcal{X}$, we have
\begin{align}
    \label{eq:offset-complexity-population-to-empirical}
    \mathfrak{R}^{\text{off}}_{n}(P_{X}, \mathcal{H}, \gamma)
    &\leq \E_{S_{n}^{X}}\left[\mathfrak{R}^{\text{off}}(S_{n}^{X}, \mathcal{H},
       \gamma)\right].
\end{align}
Thus, upper bounds on \emph{empirical} offset Rademacher complexity
$\mathfrak{R}^{\text{off}}_{n}(S_{n}^{X}, \mathcal{H}, \gamma)$
imply corresponding upper bounds on the offset Rademacher complexity.
Let us now state a bound on $\mathfrak{R}^{\text{off}}_{n}(S_{n}^{X}, \mathcal{H}, \gamma)$
for sparse linear classes, which will be used to yield sharp bounds for the
examples considered in this section.
\begin{lemma}
  \label{lemma:sparse-offset-complexity}
  For any $w \in \R^{d}$ let $\|w\|_{0}$ denote the number of non-zero
  coordinates of $w$. Denote a class of $k$-sparse linear predictors
  by
  $$
    \mathcal{H}_{\text{lin}}^{d,k}
    =
    \{\ip{w}{\cdot} : w \in \R^{d}, \|w\|_{0} \leq k\}.
  $$
  Let $S_{n}^{\Phi} = (\Phi_{i})_{i=1}^{n}$, where $\Phi_{i} \in \R^{d}$ are
  arbitrary. Then, for any $\gamma > 0$ we have
  $$
    \mathfrak{R}^{\text{off}}(S_{n}^{\Phi}, \mathcal{H}_{\text{lin}}^{d,k}, \gamma)
    \lesssim
    \frac{1}{\gamma}\log\left(\frac{ed}{k}\right)\frac{k}{n}.
  $$
\end{lemma}
The above lemma is proved in Section~\ref{sec:proof-of-sparsity-lemma} via a
direct argument involving comparison inequalities for Rademacher and Gaussian
chaos. As an immediate consequence, let us state the
following corollary that will simplify the exposition of the
applications to follow.

\begin{corollary}
  \label{corollary:sparsity}
  Let $\mathcal{G} = \{g_{1}, \dots, g_{m}\}$ denote a finite class of
  arbitrary functions mapping $\mathcal{X}$ to $\R$.
  For any positive integer $k \in \{1, \dots, m\}$ define the function class
  containing $k$-sparse linear combinations of elements of $\mathcal{G}$ by
  $$
    \mathcal{G}_{\text{lin}}^{k}
    =
    \left\{
      g_{w}(\cdot) = \sum_{i=1}^{m} w_{i} g_{i}(\cdot)
      :
      w \in \R^{d} \text{ and }
      \|w\|_{0} \leq k
    \right\}.
  $$
  Let $k_{1},k_{2} \in \{1, \dots, m\}$, $\mathcal{F} =
  \mathcal{G}_{\text{lin}}^{k_{1}}$, and fix any $g^{\star} \in
  \mathcal{G}_{\text{lin}}^{k_{2}}$. Then, for any distribution $P_{X}$
  supported on $\mathcal{X}$ and for any $\gamma > 0$ we have
  $$
    \mathfrak{R}^{\text{off}}_{n}(
      P_{X}, \operatorname{star}(\mathcal{F} - g^{\star}),
    \gamma)
    \lesssim
    \frac{1}{\gamma}\log\left(\frac{em}{(k_{1} + k_{2})}\right)\frac{(k_{1} +
    k_{2})}{n}.
  $$
\end{corollary}
\begin{proof}
  Let $k = k_{1} + k_{2}$ and note that
  $\operatorname{star}(\mathcal{F} - g^{\star}) \subseteq
  \mathcal{G}_{\text{lin}}^{k}$. Hence, the bound
  \eqref{eq:offset-complexity-population-to-empirical} yields
  \begin{equation}
    \label{eq:sparsity-corollary-first-inequality}
    \mathfrak{R}^{\text{off}}_{n}(
      P_{X}, \operatorname{star}(\mathcal{F} - g^{\star}),
    \gamma)
    \leq
    \mathfrak{R}^{\text{off}}_{n}(
      P_{X}, \mathcal{G}_{\text{lin}}^{k},
    \gamma)
    \leq \E_{S_{n}^{X}}\left[\mathfrak{R}^{\text{off}}(S_{n}^{X},
      \mathcal{G}_{\text{lin}}^{k}, \gamma)\right].
  \end{equation}
  For any sample $S_{n}^{X}$ and any $i=1,\dots,n$ define $\Phi^{X}_{i} \in \R^{m}$
  by $(\Phi^{X}_{i})_{j} = g_{j}(X_{i})$. Then, for any $w \in \R^{d}$ and $g_{w} =
  \sum_{i=1}^{m} w_{i}g_{i}$ we have $g_{w}(X_{i}) = \sum_{j=1}^{m}
  w_{j}g_{j}(X_{i}) = \langle w, \Phi^{X}_{i} \rangle$. Hence, letting
  $S_{n}^{\Phi}(S_{n}^{X}) = (\Phi_{i}^{X})_{i=1}^{n}$ and applying
  Lemma~\ref{lemma:sparse-offset-complexity} yields
  $$
    \mathfrak{R}^{\text{off}}(S_{n}^{X}, \mathcal{G}_{\text{lin}}^{k}, \gamma)
    =
    \mathfrak{R}^{\text{off}}(S_{n}^{\Phi}(S_{n}^{X}), \mathcal{F}_{\text{lin}}^{m, k}, \gamma)
    \lesssim
    \frac{1}{\gamma}\log\left(\frac{em}{k}\right)\frac{k}{n}.
  $$
  Plugging in the above inequality into
  \eqref{eq:sparsity-corollary-first-inequality} completes the proof.
\end{proof}

We now turn to the example applications.

\subsection{Model Selection Aggregation}
\label{sec:model-selection-aggregation}

In a model selection aggregation problem, we are given a finite
dictionary $\mathcal{G} = \{g_{1}, \dots, g_{m}\}$ of functions mapping
$\mathcal{X}$ to $[-b, b]$. Given a sample $S_{n} = (X_{i}, Y_{i})_{i=1}^{n}$,
a statistical estimator $\widehat{f}$ aims to construct a new function such
that the excess risk $\mathcal{E}(\widehat{f}, \mathcal{G})$ is small with high
probability.

In what follows, we consider loss functions
$\ell : [-b, b] \times [-b, b] \to [0, \infty)$ that
are $C_{b}$-Lipschitz and $\gamma$-strongly convex in the first coordinate.
More precisely, we assume that for any $y,y_{1},y_{2} \in [-b, b]$ we have
$\abs{\ell(y_{1}, y) - \ell(y_{2}, y)} \leq C_{b}\abs{y_{1} - y_{2}}$ and for
any $\lambda \in [0, 1]$ we have
$\ell(\lambda y_{1} + (1-\lambda)y_{2}, y) \leq \lambda \ell(y_{1}, y) +
(1-\lambda) \ell(y_{2}, y) - \frac{\gamma}{2} \lambda(1-\lambda)(y_{1} - y_{2})^{2}$.

An identical setup to the one described above was recently treated by
\citet*{lecue2014optimal, wintenberger2017optimal}. Optimal model selection
aggregation rates $\gamma^{-1}C_{b}^{2}\log(m/\delta)/n$
were obtained therein for the $Q$-aggregation and online Bernstein
aggregation procedures. Below, we show how the offset Rademacher complexity
analysis yields the same rates for two other estimators: Audibert's star
algorithm and the midpoint estimator.

\paragraph{Audibert's Star Algorithm.}
The star algorithm due to \citep{audibert2008progressive} is defined by
$$
  \widehat{f}^{\text{(star)}}
  = \operatorname{argmin}_{f \in \mathcal{G}, \lambda \in [0, 1]}
  R_{n}(\lambda \widehat{f}^\text{(ERM)} + (1-\lambda)f),
  \,\,\text{where}\,\,
  \widehat{f}^{\text{(ERM)}}
  = \operatorname{argmin}_{f \in \mathcal{G}}
  R_{n}(f).
$$

Generalizing an argument of \citet*[Lemma 1]{liang2015learning}, the recent
work \citet*[Proposition 5]{vijaykumar2021localization} shows that
$\widehat{f}^{\text{(star)}}$ satisfies the $(\mathcal{G}, \ell, 0,
c\gamma)$-deterministic offset condition, where $c > 0$ is some universal
constant.

In the view of Corollary~\ref{corollary:sparsity}, the range of the
star estimator $\widehat{f}^{\text{(star)}}$ is equal to
$\{\lambda g + (1-\lambda)g' : g,g' \in \mathcal{G}, \lambda \in [0,1]\}
\subseteq \mathcal{G}_{\text{lin}}^{2}$. Thus, combining
Theorem~\ref{thm:rademacher-bound} (see also
Remark~\ref{remark:lipschitz-parameter-size}) and
Corollary~\ref{corollary:sparsity} yields, for any $\delta \in
(0,1)$ with probability at least $1-\delta$
$$
  \mathcal{E}(f^{\text{(star)}}, \mathcal{G})
  \lesssim \gamma^{-1}C_{b}^{2}\frac{\log(m/\delta)}{n}.
$$

\paragraph{Midpoint Estimator.}
Let $c_{1} > 0$ be some sufficiently large universal constant (as elaborated in
the proof of Lemma~\ref{lemma:midpoint-estimator-is-offset}).
For any $\delta \in (0,1)$, the midpoint estimator is defined by
$$
  \widehat{f}_{\delta}^{\text{(mid)}}
  = \operatorname{argmin}_{f \in \mathcal{G}_{\delta, c_{1}}(S_{n})}
  R_{n}\bigg(\frac{\widehat{f}^\text{(ERM)} + f}{2}\bigg),
$$
where $\widehat{f}^{\text{(ERM)}} =
\widehat{f}^{\text{(ERM)}}(S_{n})$ is any function in $\mathcal{G}$ that
minimizes the empirical risk $R_{n}(\cdot)$ (induced by the sample $S_{n}$)
and the set $\mathcal{G}_{\delta, c_{1}}(S_{n})$ is a random (data-dependent) set of
\emph{almost empirical risk minimizers} defined by
\begin{align*}
  &\mathcal{G}_{\delta, c_{1}}(S_{n})
  = \{g \in \mathcal{G}
  : R_{n}(g) \leq R_{n}(f^{\text{(ERM)}})
  + c_{1}C_{b}d_{\delta, n}(\widehat{f}^{\text{(ERM)}}, g)\}
\end{align*}
with the empirical distance function $d_{\delta, n}$ given by, for any functions $g,g'$:
$$
  d_{\delta, n}(g, g') = \sqrt{
    \frac{n^{-1}\sum_{i=1}^{n}(g(X_{i}) - g'(X_{i}))^{2}\cdot\log(2m/\delta)}
    {n}
  }
  + \frac{b\log(2m/\delta)}{n}.
$$
In the context of model selection aggregation, the idea of applying empirical
risk minimization over some set preselected set of almost minimizers goes back to
\citet*{lecue2009aggregation}. For the recent use of midpoint procedures in
statistical literature, see, for example,
\citep*{mendelson2019unrestricted, bousquetfast, mourtada2021distribution}.

Since $\widehat{f}^{\text{(mid)}}$ outputs $2$-sparse convex combinations of
elements of the dictionary $\mathcal{G}$, similarly to the above analysis of
Audibert's star algorithm, it is enough to establish that
$\widehat{f}^{\text{(mid)}}$ satisfies the offset condition. For the midpoint
estimator, this fact is already implicit in the proofs of
\citet*{pmlr-v134-puchkin21a} in the context of active learning.
While, admittedly, the direct analysis of the midpoint estimator is no more
difficult than establishing the below lemma, for exposition purposes,
let us demonstrate that $\widehat{f}^{\text{(min)}}$ does indeed satisfy the
offset condition.
\begin{lemma}
  \label{lemma:midpoint-estimator-is-offset}
  Fix any $\delta \in (0,1)$ and any distribution $P$ supported on
  $\mathcal{X} \times [-b, b]$.
  In the setup described above, the estimator $\widehat{f}_{\delta}^{(mid)}$
  satisfies the $(\mathcal{G}, \ell, \varepsilon, (64)^{-1}\gamma)$-offset
  condition for the distribution $P$,
  with $\varepsilon(\delta) \lesssim C_{b}^{2}\gamma^{-1}\log(2m/\delta)/n$.
\end{lemma}
The proof is deferred to
Appendix~\ref{sec:proof-of-lemma-midpoint-estimator-is-offset}.
An immediate consequence of the above lemma,
via an application of Theorem~\ref{thm:rademacher-bound}
(with $\delta_{1} = \delta_{2} = \delta/2$) and
Corollary~\ref{corollary:sparsity} is that for any $\delta \in (0,1)$ with
probability at least $1-\delta$ the following holds:
$$
  \mathcal{E}(\widehat{f}^{\text{(mid)}}_{\delta}, \mathcal{G})
  \lesssim \gamma^{-1}C_{b}^{2}\frac{\log(4m/\delta)}{n}.
$$
\subsection{Iterative Regularization}
\label{sec:mirror-descent-example}
The idea of iterative regularization is to apply some
optimization procedure to the \emph{unregularized} empirical risk function
$R_{n}(\cdot)$ and induce a regularizing effect by early stopping. Thus, the
number of iterations performed acts as a regularization parameter, in a similar
way that the size of penalty acts as a regularization parameter for penalized
procedures based on empirical risk minimization.
Iterative regularization schemes are actively studied since they have a built-in
warm-restart feature: obtaining a new model only costs one iteration of the
optimization algorithm, usually amounting to a gradient descent or stochastic
gradient descent update. In contrast, for explicitly penalized procedures,
obtaining new models (corresponding to different regularization parameters)
amount to solving a new optimization problem. Let us demonstrate an example
of how a general family of such algorithms fit into the framework of offset Rademacher complexity.

Let $\mathcal{X}$ be a compact subset of $\R^{d}$. In this section, we fix the
set of reference functions to be $\mathcal{G} = \{f_{w}(\cdot) = \langle w,
\cdot \rangle : w \in G \subset \R^{d}\}$, where the set $G$ is arbitrary.
Denote any population risk minimizer in $\mathcal{G}$ by $g^{\star} =
f_{w^{\star}}$, where $w^{\star} \in G$. Further, for any $w \in \R^{d}$, let
$R(w) = R(f_{w})$ and $R_{n}(f_{w}) = R_{n}(w)$.

We consider a family of mirror descent algorithms
\citep*{nemirovsky1983problem, beck2003mirror} that admit the more frequently
studied gradient descent procedure as a special case.
Let $\mathcal{D} \subseteq \R^{d}$ be an open and convex set.
Let $\psi : \mathcal{D} \to \R^{d}$ denote a continuously differentiable
strictly convex function whose gradient diverges
at the boundary of $\mathcal{D}$. We call such a function a \emph{mirror map}.
The associated \emph{Bregman divergence} $D_{\psi} : \mathcal{D} \times \mathcal{D}
\to \mathbb{R}$ is defined by $D_{\psi}(x,y) = \psi(x) - \psi(y) - \langle
\nabla \psi (y), x - y\rangle$; note that for any $x, y \in \mathcal{D}$ we
have $D_{\psi}(x, y) \geq 0$ due to the convexity of $\psi$.
In \emph{continuous-time}, the mirror descent algorithm is defined by the following
differential equation, where $t \geq 0$ is the time parameter:
\begin{equation}
  \label{eq:mirror-descent-flow}
  \frac{d}{dt}w_{t} = -\left(\nabla^{2}\psi(w_{t})\right)^{-1}\nabla
  R_{n}(w_{t}).
\end{equation}
We now present an argument due to
\citet*{vaskevicius2020statistical}, where it was shown that early-stopped
mirror descent algorithms satisfy the offset condition.
\begin{lemma}
  \label{lemma:mirror-descent}
  As defined above, let $\mathcal{G}$ be any reference class of linear functions
  and denote $g^{\star} = f_{w^{\star}}$. Let $\ell$ be a differentiable and
  $\gamma$-strongly convex loss function in its first argument (cf.
  Section~\ref{sec:model-selection-aggregation}).
  Fix an arbitrary initialization point
  $w_{0} \in \R^{d}$ and let $(w_{t})_{t > 0}$ be generated by
  the mirror descent flow \eqref{eq:mirror-descent-flow}. Then, for any
  $\varepsilon > 0$ there exists a
  (random) stopping time $t^{\star} = t^{\star}(S_{n}, w^{\star}, w_{0})$
  such that the following three deterministic conditions hold:
  \begin{enumerate}
    \item The stopping time satisfies the deterministic bound
      $t^{\star} \leq 2D_{\psi}(w^{\star}, w_{0})/\varepsilon$;
    \item The early-stopped iterate $w_{t^{\star}}$ satisfies
      $w_{t^{\star}} \in \{w \in \R^{d} : D_{\psi}(w^{\star}, w) \leq
      D_{\psi}(w^{\star}, w_{0}\}$;
    \item The estimator $\widehat{f} = f_{w_{t^{\star}}}$
      satisfies the $(\mathcal{G}, \ell, \varepsilon,
      \frac{\gamma}{2})$-deterministic
      offset condition.
  \end{enumerate}
\end{lemma}

\begin{proof}
  For any $t \geq 0$, let $\delta(t) = R_{n}(w_{t}) - R_{n}(w^{\star}) +
  \frac{\gamma}{2}\sum_{i=1}^{n}(f_{w_{t}}(X_{i}) - f_{w^{\star}}(X_{i}))^{2}$.
  Let $t^{\star} \coloneqq \inf\{t \geq 0 : \delta(t) \leq \varepsilon \}$
  A direct computation shows the following well-known identity:
  $-\frac{d}{dt}D_{\psi}(w^{\star}, w_{t}) = \langle - R_{n}(w_{t}), w^{\star} -
  w_{t}\rangle$. By the $\gamma$-strong convexity assumption, it hence follows
  that for any $t \geq 0$ we have $-\frac{d}{dt}D_{\psi}(w^{\star}, w_{t}) \geq
  \delta(t)$. Integrating both sides, it follows that the following infimum is
  well defined and it satisfies all the conditions of this theorem:
  $t^{\star} = \inf\{0 \leq t \leq 2D_{\psi}(w^{\star}, w_{0})/\varepsilon :
  \delta(t) \leq \varepsilon \}$.
\end{proof}

Observe that the above argument only involves the tools from convex
optimization, yet Theorem~\ref{thm:rademacher-bound} readily implies
probabilistic performance bounds for the estimator considered above.
Condition 1 in the above lemma establishes a
statistical-computational trade-off. Condition 2 determines the range of the
early-stopped estimator. Condition 3 shows that the early-stopped mirror
descent estimator can be analyzed via offset Rademacher complexities; indeed,
this is the only known approach for obtaining sharp guarantees for this general
class of iterative regularization schemes (see
\citep*{vaskevicius2020statistical} for further discussion and for
discrete-time results).
For more examples and further background on iterative
regularization, see, for example, \citep*{buhlmann2003boosting, yao2007early,
raskutti2014early, lin2016iterative, wei2019early}.


\section{Proof of Theorem~\ref{thm:rademacher-bound}}
\label{sec:proof-of-rademacher-bound}

Recall that $P$ denotes the underlying distribution of $(X,Y)$ and let
$P_{n}$ denote its empirical counterpart supported on the sample $S_{n}$ so
that
\begin{align*}
  &P\ell = \mathbf{E}_{(X,Y)\sim P}[\ell(X, Y)]\quad\text{and}\quad
  P_{n}\ell = \frac{1}{n}\sum_{i=1}^{n}\ell(X_{i}, Y_{i})
  \quad\text{for any function }\ell : \mathcal{X} \times \mathcal{Y} \to
  \mathbb{R};
  \\
  &Ph = \mathbf{E}_{X \sim P_{X}}[h(X)]\quad\text{and}\quad
  P_{n}h = \frac{1}{n} \sum_{i=1}^{n}h(X_{i})
  \quad\text{for any function }h : \mathcal{X} \to \mathbb{R}.
\end{align*}
With the above notation we have $R(f) = P\ell_{f}$ and $R_{n}(f) = P_{n}\ell_{f}$.
Denote the event
$$
  E_{\delta_{2}} = \{P_{n}\ell_{\widehat{f}} - P_{n}\ell_{g^{\star}} \leq -\gamma
  P_{n}(\widehat{f} - g^{\star})^{2} + \varepsilon(\delta_{2})\}
$$
Since $\widehat{f}$ satisfies the $(\mathcal{G}, \ell,
\varepsilon, \gamma)$-offset condition we have $\P(E_{\delta_{2}}) \geq 1 -
\delta_{2}$; on $E_{\delta_{2}}$ we have
\begin{align*}
  P\ell_{\widehat{f}} - P\ell_{g^{\star}}
  &=
  (P - P_{n})(\ell_{\widehat{f}} - \ell_{g^{\star}})
  + P_{n}(\ell_{\widehat{f}} - \ell_{g^{\star}})
  \\
  &\leq
  (P - P_{n})(\ell_{\widehat{f}} - \ell_{g^{\star}})
  - \gamma P_{n}(\widehat{f} - g^{\star})^{2} + \varepsilon(\delta_{2})
  \\
  &\leq
  \underbrace{\sup_{f \in \mathcal{F}}\left\{
    (P-P_{n})(\ell_{f} - \ell_{g^{\star}})
    - \gamma P_{n}(f - g^{\star})^{2}
  \right\}}_{\coloneqq Z}
  + \varepsilon(\delta_{2}).
\end{align*}
The rest of the proof  is structured as follows:
\begin{enumerate}
  \item We first symmetrize a suitably rearranged Laplace transform of the
    empirical offset process $Z$. Since for $\lambda \geq 0$ the map
    $x \mapsto e^{\lambda x}$ is convex and non-decreasing, this step of the
    proof follows via standard arguments.
  \item Next, we apply Talagrand's Contraction Lemma to the symmetrized
    offset empirical process. This step turns our process into a multiplier-type
    process of Proposition~\ref{prop:multiplier-concentration}.
  \item We conclude the proof via an application of
    Proposition~\ref{prop:multiplier-concentration}, which yields a
    Bernstein-type upper bound on the moment generating function of
    $Z - \mathfrak{R}^{\text{off}}_{n}(\operatorname{star}(\mathcal{H}), \gamma')$, for a
    suitably defined constant $\gamma' > 0$.
    The desired tail bound then follows via Markov's inequality.
\end{enumerate}

\begin{remark}
Our proof strategy is inspired by the work of \citet{lecue2014optimal}, where
symmetrization and contraction arguments are also performed on the Laplace
transform of the empirical process of interest. The contraction step is
needed there to make the corresponding complexity measure linear in the model
parameters so that the supremum over a convex hull is attained a vertex.
In contrast, we need to apply the contraction step to put us in the setting of
Proposition~\ref{prop:multiplier-concentration}.
\end{remark}

\paragraph{Symmetrization step.}
We begin by rewriting the random variable $Z$ as follows:
\begin{align}
  Z
  &=
  \sup_{f \in \mathcal{F}}\left\{
      (P - P_{n})\left(
        \ell_{f} - \ell_{g^{\star}}
      \right)
      - \gamma P_{n}(f - g^{\star})^{2}
  \right\}
  \\
  &=
  \sup_{f \in \mathcal{F}}\left\{
      (P - P_{n})\left(
        \ell_{f} - \ell_{g^{\star}}
        + \frac{3\gamma}{4}(f - g^{\star})^{2}
      \right)
      - \frac{\gamma}{4} P_{n}(f - g^{\star})^{2}
      - \frac{3\gamma}{4} P (f - g^{\star})^{2}
  \right\},
  \label{eq:symmetrization-step-first-rearrangement}
\end{align}
where in the last equation above we have added and subtracted
$(3\gamma/4) P (f - g^{\star})^{2}$. For any function $f \in \mathcal{F}$
introduce a shorthand notation
$$
  \phi_{f} : \mathcal{X} \times \mathcal{Y} \to \mathbb{R}
  \enskip\text{such that}\enskip
  \phi_{f}(X, Y) = \ell_{f}(X, Y) - \ell_{g^{\star}}(X, Y)
  + \frac{3\gamma}{4}(f(X) - g^{\star}(X))^{2}.
$$
Let $S_{n}' = (X_{i}', Y_{i}')_{i=1}^{n}$
denote an independent copy of $S_{n} = (X_{i}, Y_{i})_{i=1}^{n}$
and denote $\mathbf{E}'$ as a shorthand notation for expectation computed with
respect to $S_{n}'$ only, conditionally on all other random variables.
Let $P_{n}'$ denote a counterpart to $P_{n}$ with the sample $S_{n}$
replaced by $S_{n}'$.
Carrying on from equation \eqref{eq:symmetrization-step-first-rearrangement}
we can rewrite $Z$ as follows:
\begin{align}
  Z &= \sup_{f \in \mathcal{F}}\left\{
    (P - P_{n})\phi_{f}
      - \frac{\gamma}{4} P_{n}(f - g^{\star})^{2}
      - \frac{3\gamma}{4} P (f - g^{\star})^{2}
  \right\}
  \\
  &=
  \sup_{f \in \mathcal{F}}\left\{
    (P - P_{n})\phi_{f}
      - \frac{\gamma}{4} P_{n}(f - g^{\star})^{2}
      - \frac{\gamma}{4} P (f - g^{\star})^{2}
      - \frac{2\gamma}{4} P (f - g^{\star})^{2}
  \right\}
  \\
  \label{eq:pre-symmetrization-step}
  &=
  \sup_{f \in \mathcal{F}}\left\{
    (\mathbf{E}'P_{n}' - P_{n})\phi_{f}
      - \frac{\gamma}{4} P_{n}(f - g^{\star})^{2}
      - \frac{\gamma}{4} \mathbf{E}'P_{n}' (f - g^{\star})^{2}
      - \frac{2\gamma}{4} P (f - g^{\star})^{2}
  \right\}.
\end{align}
Observe that in the above equation we have left the term $(2\gamma/4)P(f -
g^{\star})$ unchanged. This is needed to put us in the setting of
Proposition~\ref{prop:multiplier-concentration}, as we shall see below.

Let us now introduce a sequence of $n$
independent Rademacher (symmetric and $\{\pm1\}$ valued) random variables
$\sigma_{i}$ and let $\mathbf{E}_{\sigma}$ denote expectation with $\sigma_{1},
\dots, \sigma_{n}$ only, conditionally on all other random variables.
Let $P_{n}^{\sigma}$ denote the symmetrized empirical measure
so that for any function $\ell : \mathcal{X} \times \mathcal{Y} \to \mathbb{R}$
and any function $h : \mathcal{X} \to \mathbb{R}$ we have
\begin{align}
  P_{n}^{\sigma}\ell = \frac{1}{n}\sum_{i=1}^{n}\sigma_{i}\ell(X_{i}, Y_{i})\quad
  \text{and}\quad
  P_{n}^{\sigma}h = \frac{1}{n}\sum_{i=1}^{n}\sigma_{i}h(X_{i}).
\end{align}
For $\lambda > 0$ the map $x \mapsto e^{\lambda x}$ is convex and
non-decreasing; hence, for any $\lambda > 0$, using the identity
\eqref{eq:pre-symmetrization-step}, we can proceed to symmetrize the Laplace transform of $Z$
as follows:
\begin{align}
  \mathbf{E}\exp(\lambda Z)
  &\leq
  \mathbf{E}\mathbf{E}'
  \exp\bigg(\lambda\sup_{f \in \mathcal{F}}\bigg\{
      (P_{n}' - P_{n})\phi_{f}
      - \frac{\gamma}{4}P_{n} (f - g^{\star})^{2}
      \\
      &\quad\quad
      - \frac{\gamma}{4}P_{n}' (f-g^{\star})^{2} -
      \frac{2\gamma}{4}P(f - g^{\star})^{2}
  \bigg\}\bigg)
  \\
  &\leq
  \mathbf{E}\mathbf{E}_{\sigma}
  \exp\left(2\lambda\sup_{f \in \mathcal{F}}\left\{
      P_{n}^{\sigma}\phi_{f}
      - \frac{\gamma}{4}P_{n} (f - g^{\star})^{2}
      - \frac{\gamma}{4}P (f - g^{\star})^{2}
  \right\}\right).
  \label{eq:rad-bound-symmetrized}
\end{align}
Notice that the above moment generating function is almost of the form
that can be bounded via Proposition~\ref{prop:multiplier-concentration}.
It remains to replace the term $P_{n}^{\sigma}\phi_{f}$ with a term
$\rho P_{n}^{\sigma}(f - g^{\star})$, for some constant
$\rho$. This is the aim of the contraction step of this proof, which follows
below.

\paragraph{Contraction step.}
Recall that by the assumptions of this theorem, there exists some constant
$C_{b}$ such that for any $f, f' \in \mathcal{F},  x \in \mathcal{X}, y \in \mathcal{Y}$
we have
$$
  \abs{\ell_{f}(x, y) - \ell_{f'}(x, y)}
  \leq C_{b}\abs{f(x) - f'(x)}.
$$
In particular, for any $f,f' \in \mathcal{F}$ and any $x \in \mathcal{X}, y \in
\mathcal{Y}$ we have
\begin{align}
  \abs{\phi_{f}(x, y) - \phi_{f'}(x,y)}
  &=
  \left|\ell_{f}(x, y) + \frac{3\gamma}{4}(f(x) - g^{\star}(x))^{2} -
  \ell_{f'}(x,y) - \frac{3\gamma}{4}(f'(x) - g^{\star}(x))\right|
  \\
  &\leq
  C_{b}\abs{
    f(x) - f'(x)
  }
  + \frac{3\gamma}{4}\abs{(f(x) - f'(x))(f(x) + f'(x) - 2g^{\star}(x))}
  \\
  &\leq
  (C_{b} + 3\gamma b)\abs{f(x) - f'(x)}
  \\
  &=
  (C_{b} + 3\gamma b)\abs{(f(x) - g^{\star}(x)) - (f'(x) - g^{\star}(x))}.
\end{align}
Hence, applying Talagrand's contraction inequality \citep[Theorem 4.12]{ledoux1991probability}
(conditionally on the sample $S_{n}$) with the set $T_{S_{n}}$ and contraction mappings
$\phi_{S_{n}}^{(i)}$:
\begin{align}
  T_{S_{n}} &=
  \{((f - g^{\star})(X_{1}), \dots, (f - g^{\star})(X_{n}))^{\mathsf{T}} : f \in \mathcal{H}\},
  \\
  \phi_{S_{n}}^{(i)}(t_{i}) &=
  \left(2C_{b} + 6\gamma b\right)^{-1}
  \cdot 2
  \left(
    \ell(t_{i} + g^{\star}(X_{i}), Y_{i}) - \ell_{g^{\star}}(X_{i}, Y_{i})
    - \frac{3\gamma}{4}t_{i}^{2}
  \right),
\end{align}
we may proceed upper bounding \eqref{eq:rad-bound-symmetrized} as follows
(cf.\ \citep*[Eq. (3.11)]{lecue2014optimal}):
\begin{align}
  &\mathbf{E}\exp\left(\lambda Z\right)
  \\
  &\leq
  \mathbf{E}\mathbf{E}_{\sigma}
  \exp\left(\lambda\sup_{f \in \mathcal{F}}\left\{
      P_{n}^{\sigma}2\phi_{f}
      - \frac{\gamma}{2}P_{n} (f - g^{\star})^{2}
      - \frac{\gamma}{2}P (f - g^{\star})^{2}
  \right\}\right)
  \\
  &\leq
  \mathbf{E}\mathbf{E}_{\sigma}
  \exp\left(\lambda\sup_{f \in \mathcal{F}}\left\{
      (2C_{b} + 6\gamma b)P_{n}^{\sigma}(f - g^{\star})
      - \frac{\gamma}{2}P_{n} (f - g^{\star})^{2}
      - \frac{\gamma}{2}P (f - g^{\star})^{2}
  \right\}\right)
  \\
  &=
  \mathbf{E}\mathbf{E}_{\sigma}
  \exp\left(\lambda\sup_{h \in \mathcal{H}}\left\{
      (2C_{b} + 6\gamma b)P_{n}^{\sigma}h
      - \frac{\gamma}{2}P_{n}h^{2}
      - \frac{\gamma}{2}Ph^{2}
  \right\}\right)
  \\
  &\leq
  \mathbf{E}\mathbf{E}_{\sigma}
  \exp\bigg(\frac{\lambda}{n}\cdot \underbrace{
      n\sup_{h \in \operatorname{star}(\mathcal{H})}\left\{
      (2C_{b} + 6\gamma b)P_{n}^{\sigma}h
      - \frac{\gamma}{2}P_{n}h^{2}
      - \frac{\gamma}{2}Ph^{2}
  \right\}}_{\coloneqq U}
  \bigg),
\end{align}
where in the penultimate line we introduced $\mathcal{H} = \{f - g^{\star} : f
\in \mathcal{F}\}$, and in the last step the inequality comes from replacing
$\mathcal{H}$ by $\operatorname{star}(\mathcal{H}) = \{\lambda h : h \in \mathcal{H}, \lambda
\in [0, 1]$\}.

We will now show that the random variable $U$ is a supremum of an offset multiplier
process satisfying the conditions of Proposition~\ref{prop:multiplier-concentration}.
Let $\zeta_{i} = (2C_{b} + 6\gamma b) \sigma_{i}$ and denote the
distribution of $\zeta$ by $P_{\zeta}$. Then, for any $h \in \mathcal{H}$ and
for $(X, \zeta)$ distributed
according to the product distribution $P_{X} \otimes P_{\zeta}$, we have
$\E[\zeta h(X)] = 0$. Therefore,
\begin{align}
  U &=
      n\cdot\sup_{h \in \operatorname{star}(\mathcal{H})}\left\{
      (2C_{b} + 6\gamma b)P_{n}^{\sigma}h
      - \frac{\gamma}{2}P_{n}h^{2}
      - \frac{\gamma}{2}Ph^{2}
    \right\}
  \\
  &= \sup_{h \in \operatorname{star}(\mathcal{H})}\left\{
      \sum_{i=1}^{n} \zeta_{i}h(X_{i}) -
      \E_{(X, \zeta) \sim P_{X} \otimes P_{\zeta}}[\zeta h(X)]
      - \frac{\gamma}{2} h(X_{i})^{2}
      - \frac{\gamma}{2} \E_{X \sim P_{X}} h(X)^{2}
    \right\}.
\end{align}
Hence, the moment generating function of the random variable $U$ can be bounded
via Proposition~\ref{prop:multiplier-concentration}, taking $P_{(X,\zeta)} =
P_{X} \otimes P_{\zeta}$.

\paragraph{Concluding the proof.}
Let $c_{3} > 0$ be some universal constant such that
\begin{equation}
  \eta = 8((2C_{b} + 6\gamma b)^{2}(\gamma/2)^{-1} + (\gamma/2)4b^2)
  \leq c_{3} (\gamma^{-1}C_{b}^{2} + bC_{b} + \gamma b^{2}).
\end{equation}
Relabelling $\lambda/n$ by $\lambda$
and applying Proposition~\ref{prop:multiplier-concentration}
to the random variable $U$, the following holds for any $\lambda \in (0,
1/\eta)$:
\begin{equation}
  \label{eq:rad-sub-gamma}
  \log \mathbf{E}\exp\left(\lambda((nZ) - \E\E_{\sigma}U)\right)
  \leq
  \log \mathbf{E}\E_{\sigma}\exp\left(\lambda(U - \E\E_{\sigma}U)\right)
  \leq
  \frac{\lambda^{2}\eta\E\mathbf{E}_{\sigma}U}{2(1 - \eta\lambda)}.
\end{equation}
The desired tail bound now follows via standard arguments that we sketch below.
By \citep*[Section 2.4]{boucheron2013concentration}, the upper bound
\eqref{eq:rad-sub-gamma} shows that the random variable
$nZ - \E\E_{\sigma}U$ is sub-gamma on the right-tail with variance proxy
$\eta \E\E_{\sigma} U$ and scale parameter $\eta$. Hence,
via Markov's inequality, for any $\delta_{1} \in (0,1]$ we have
\begin{equation}
  \P \left(
    nZ - \E\E_{\sigma}[U] \geq
    \sqrt{2\eta\E\E_{\sigma}[U]\log(\delta^{-1})}
    + \eta\log(\delta_{1}^{-1})
  \right)
  \leq \delta_{1}.
\end{equation}
Subtracting $\E\E_{\sigma}[U]$ from both sides of the inequality defining the
event inside $\P(\cdot)$ and optimizing the quadratic function in
$\sqrt{\E\E_{\sigma}[U]}$, we deduce that
\begin{align*}
  \delta_{1}
  &\geq
  \P \left(
    nZ - 2\E\E_{\sigma}[U] \geq \sqrt{2\eta\E\E_{\sigma}[U]\log(\delta_{1}^{-1})} -
    \E\E_{\sigma}[U'] + \eta \log(\delta_{1}^{-1})
  \right) \\
  &\geq
  \P \left(
    nZ - 2\E\E_{\sigma}[U] \geq \sup_{x\in\R}\left\{\sqrt{2\eta x\log(\delta^{-1})} -
    x^{2}\right\} + \eta\log(\delta_{1}^{-1})
  \right) \\
  &=
  \P \left(
    nZ - 2\E\E_{\sigma}[U] \geq (3/2) \eta \log(\delta_{1}^{-1})
  \right).
\end{align*}
Thus, denoting the event
$$
  E_{\delta_{1}} = \{
    nZ - 2\E\E_{\sigma}[U] \leq (3/2) \eta \log(\delta_{1}^{-1})
  \}
$$
we have $\P(E_{\delta_{1}}) \geq 1 - \delta_{1}$.
Finally, observe that
\begin{align}
  \mathbf{E}_{S_{n}}\mathbf{E}_{\sigma} U
  &= n(2C_{b} + 6\gamma b)
  \mathfrak{R}^{\text{off}}_{n}\left(P_{X}, \operatorname{star}(\mathcal{H}),
  \frac{\gamma}{2}\cdot(2C_{b} + 6\gamma b)^{-1}\right)
  \\
  &\leq 74 \cdot n(C_{b} + \gamma b)
  \mathfrak{R}^{\text{off}}_{n}\left(P_{X}, \operatorname{star}(\mathcal{H}),
  \gamma\cdot(C_{b} + \gamma b)^{-1}\right).
\end{align}
The desired result follows by the union bound on the events $E_{\delta_{1}}$
and $E_{\delta_{2}}$. $\hfill\qed$


\section{Proof of Proposition~\ref{prop:multiplier-concentration}}
\label{sec:proof-of-multiplier-proposition}

Let us first discuss the key insight into our proof.
Without loss of generality, assume that the supremum in
the definition of the random variable $U$ (cf.\ \eqref{eq:dfn-multiplier})
is always attained by some function, and denote this (random) function by
$\tilde{h} = \tilde{h}(S_{n})$. The following lemma shows that
the empirical and population $L_{2}$ norms of $\tilde{h}$ are upper bounded
by $c^{-1}U$. Thus, intuitively the supremum over
$\operatorname{star}(\mathcal{H})$ in the multiplier process is computed over
a ``self-localized'' (in a random/data-dependent way) subset of
$\operatorname{star}(\mathcal{H})$.
In contrast, we remark that the classical theory of localization via fixed-point
equations proceeds by localizing the function class
$\operatorname{star}(\mathcal{H})$ by constraining it to
an \emph{explicitly} chosen subset of functions with small $L_{2}$ population
or empirical norms.

\begin{lemma}
  \label{lemma:multiplier-self-localization}
  Consider the setting of Proposition~\ref{prop:multiplier-concentration} and
  let $\tilde{h} = \tilde{h}(S_{n})$ denote a random function that attains the
  supremum of the offset multiplier process $U$ (cf.\ \eqref{eq:dfn-multiplier})
  given the sample $S_{n} = (X_{i}, \zeta_{i})_{i=1}^{n}$. That is, $\tilde{h}$ satisfies
  \begin{equation}
    \label{eq:maximizer-of-multiplier-process}
    \sum_{i=1}^{n} \left(
      \zeta_{i}\tilde{h}(X_{i}) - \E [\zeta \tilde{h}(X) \vert S_{n}]
    - \gamma \tilde{h}(X_{i})^{2}
    - \gamma \E[\tilde{h}(X)^{2} \vert S_{n}]
    \right)
    = U(S_{n}).
  \end{equation}
  Then, the following deterministic
  inequality holds for any realization of $S_{n}$:
  \begin{equation}
    \label{eq:multiplier-minimizer-norms}
    \sum_{i=1}^{n}\left(\E[\tilde{h}(X)^{2}\vert S_{n}] + \tilde{h}(X_{i})^{2}\right)
    \leq \frac{1}{\gamma}U(S_{n}).
  \end{equation}
\end{lemma}

\begin{proof}
  Fix any realization $S_{n} = (X_{i}, \zeta_{i})_{i=1}^{n}$ and in the rest of
  this proof we work conditionally on $S_{n}$.
  For any $h \in \operatorname{star}(\mathcal{H})$, define $A(h)$ and $B(h)$ as follows:
  \begin{equation}
    \label{eq:A-B-definition}
    A(h) =
    \sum_{i=1}^{n} \left(
      \zeta_{i}h(X_{i}) - \E [\zeta h(X) \vert S_{n}]
    \right),
    \quad
    B(h) =
    \gamma \sum_{i=1}\left(
      \E[h(X)^{2} \vert S_{n}] + h(X_{i})^{2}]
    \right).
  \end{equation}
  Thus, since $\tilde{h} = \tilde{h}(S_{n})$ denotes a maximizer of the offset
  multiplier process, we have
  \begin{equation}
    \label{eq:tilde-h-A-B-split}
    A(\tilde{h}) - B(\tilde{h}) = \sup_{h\in \operatorname{star}(\mathcal{H})}
    (A(h) - B(h)) = U(S_{n}).
  \end{equation}
  For any $\lambda \in [0, 1)$, let $\lambda h : x \mapsto \lambda h(x)$.
  Observe that for any $h$ and $\lambda$, the term
  $A(\lambda h)$ scales \emph{linearly} as a function of $\lambda$
  (i.e., $A(\lambda h) = \lambda A(h)$),
  while the term $B(\lambda h)$ scales \emph{quadratically}
  (i.e., $B(\lambda h) = \lambda^{2} B(h))$ as a function of $\lambda$.
  Fix any $\lambda \in [0,1)$ and note that by the definition of star-hulls,
  the function $\lambda \tilde{h}$ is in the set
  $\operatorname{star}(\mathcal{H})$. Therefore, the identity
  \eqref{eq:tilde-h-A-B-split} implies that
  \begin{equation}
    \label{eq:rescaled-A-B-tilde}
    \lambda A(\tilde{h}) - \lambda^{2}(B(\tilde{h}))
    = A(\lambda\tilde{h}) - B(\lambda\tilde{h})
    \leq \sup_{h \in \operatorname{star}(\mathcal{H})}(A(h) - B(h))
    = U(S_{n}).
  \end{equation}
  Rearranging the identity \eqref{eq:tilde-h-A-B-split} we also have
  $A(\tilde{h}) = U(S_{n}) + B(\tilde{h})$, which plugged into
  the left hand side of \eqref{eq:rescaled-A-B-tilde} yields
  \begin{equation}
    \lambda(1 - \lambda)B(\tilde{h}) \leq (1-\lambda)U(S_{n}).
  \end{equation}
  Dividing both sides by $(1-\lambda) > 0$ shows that $\lambda B(\tilde{h}) \leq
  U(S_{n})$. Since the last equation holds for any $\lambda \in [0, 1)$
  it follows that $B(\tilde{h}) \leq U(S_{n})$ which completes
  the proof of this lemma.
\end{proof}

With the above lemma in place, we are ready to prove
Proposition~\ref{prop:multiplier-concentration}. In the below proof,
we follow the standard approach for obtaining Bernstein-type concentration
bounds for the supremum of empirical processes
(see \citep*[Section 12.2]{boucheron2013concentration}).
In particular, such bounds often build on the entropy method, which in our case
appears through an application of the exponential Efron-Stein
inequality. For a survey of tail bounds on the supremum of empirical processes,
see the bibliographic remarks in \citep*[Section 12]{boucheron2013concentration}.
We now introduce some additional notation.

\begin{enumerate}
  \item Let $S_{n}^{(i)}$ be equal to the sample $S_{n}$ with the $i$-th
    element $(X_{i}, \zeta_{i})$ replaced by an independent copy
    $(X_{i}', \zeta_{i}') \sim P_{(X, \zeta)}$.
  \item For $i=1,\dots,n$, let $U_{i}' = U(S_{n}^{(i)}))$.
    Thus $U_{i}'$ is the supremum of the offset multiplier process
    computed on the sample $S_{n}^{(i)}$, which differs from $S_{n}$
    by the $i$-th sample only.
  \item Let $\E'[\cdot] = \E[\cdot \vert S_{n}]$ denote the expectation computed
    with respect to the random variables $(X_{i}', \zeta_{i}')$ only.
    In particular, we have $\E'[U] = U$.
\end{enumerate}

The exponential Efron-Stein inequality \citep*[Theorem
6.16]{boucheron2013concentration} asserts that for $\theta > 0$ and any
$\lambda \in (0, 1/\theta)$ we have
\begin{equation}
  \label{eq:exponential-efron-stein}
  \log \E e^{\lambda(U - \E U)}
  \leq \frac{\lambda \theta}{1 - \lambda \theta} \log \E e^{\lambda
  V^{+}/\theta},
  \quad\text{where}\quad
  V^{+} = \sum_{i=1}^{n} \E'[(U - U_{i}')_{+}^{2}].
\end{equation}
To complete the proof of Proposition~\ref{prop:multiplier-concentration}, it
remains to upper bound the random variable $V^{+}$. This will be achieved via a
combination of Lemma~\ref{lemma:multiplier-self-localization} and boundedness
assumptions on the function class $\mathcal{H}$ and the multipliers $\zeta$.
Indeed, let $\tilde{h} = \tilde{h}(S_{n})$ be a function that attains the
supremum in the definition of $U$ (cf.\
Lemma~\ref{lemma:multiplier-self-localization})
Then, evaluating the multiplier process defined on the sample $S^{(i)}_{n}$
with the function $\tilde{h}$ yields a lower bound on $U_{i}$.
Therefore, for $i = 1,\dots,n$ we have
$$
  U - U_{i}' \leq
    \zeta_{i}\tilde{h}(X_{i}) - \gamma \tilde{h}(X_{i})^{2}
    - \zeta_{i}'\tilde{h}(X_{i}') + \gamma \tilde{h}(X_{i}')^{2}
$$
and hence,
$$
  (U - U_{i}')_{+}^{2} \leq \left(
    \zeta_{i}\tilde{h}(X_{i}) - \gamma \tilde{h}(X_{i})^{2}
    - \zeta_{i}'\tilde{h}(X_{i}') + \gamma \tilde{h}(X_{i}')^{2}
  \right)^{2}.
$$
Noting that for any $a,b,c,d \in \R$ we have
$(a + b + c + d)^{2} \leq 4a^{2} + 4b^{2} + 4c^{2} + 4d^{2}$ (for example, by
the Cauchy-Schwarz inequality) it follows that
\begin{align*}
  \E'[(U - U_{i}')_{+}^{2}]
  &\leq
  4\E'[\zeta_{i}^{2}\tilde{h}(X_{i})^{2} + \gamma^{2}\tilde{h}(X_{i})^{4}
  + \zeta_{i}'^{2}\tilde{h}(X_{i}')^{2} + \gamma^{2}\tilde{h}(X_{i}')^{4}]
  \\
  &\leq 4\E'[(\sigma^{2} + \gamma^{2}\kappa^{2})(\tilde{h}(X_{i})^{2} +
  \tilde{h}(X_{i}')^{2})]
  \\
  &\leq 4(\sigma^{2} + \gamma^{2}\kappa^{2})(\tilde{h}(X_{i})^{2} +
  \E[\tilde{h}(X)^{2} \vert S_{n}]),
\end{align*}
where the second line follows by the boundedness assumptions and the last line
follows by noting that $\tilde{h}(X_{i})$ depends on $S_{n}$ only and
renaming $X_{i}'$ to $X$. Hence, we can now obtain an upper bound on $V^{+}$
defined in \eqref{eq:exponential-efron-stein} via
Lemma~\ref{lemma:multiplier-self-localization} as follows:
\begin{equation}
  0 \leq V^{+} \leq
  4(\sigma^{2} + \gamma^{2}\kappa^{2})
  \sum_{i=1}^{n}\left(
    \tilde{h}(X_{i})^{2} + \E[\tilde{h}(X)^{2}\vert S_{n}]
  \right)
  \leq
  4(\sigma^{2}\gamma^{-1} + \gamma\kappa^{2})U
\end{equation}
Plugging the above upper bound on $V^{+}$ into the exponential Efron-Stein
inequality \eqref{eq:exponential-efron-stein} with the choice $\theta =
4(\sigma^{2}\gamma^{-1} + \gamma\kappa^{2})$ yields,
for any $\lambda \in (0, 1/\theta)$:
\begin{align*}
  \log \E e^{\lambda(U - \E U)}
  \leq \frac{\lambda \theta}{1 - \lambda \theta} \log \E e^{\lambda U}
  = \frac{\lambda \theta}{1 - \lambda \theta} \left(
    \log \E e^{\lambda(U - \E U)}
    + \lambda \E U
  \right).
\end{align*}
Rearranging the above inequality, we obtain
\begin{equation}
  \frac{1 - 2\lambda\theta}{1 - \lambda \theta} \log \E e^{\lambda(U - \E U)}
  \leq \frac{\lambda^{2}\theta \E U}{1-\lambda \theta}.
\end{equation}
For any $\lambda \in (0, 1/(2\theta))$ we have $(1-2\lambda \theta)/(1-\lambda
\theta) > 0$, thus for $\lambda \in (0, 1/(2\theta))$ we have
\begin{equation}
  \label{eq:log-mgf-multiplier-upper-bound}
  \log \E e^{\lambda(U - \E U)}
  \leq \frac{\lambda^{2}\theta \E[U]}{1-2\lambda \theta}
  = \frac{\lambda^{2}(\eta\E U)}{2(1-\eta\lambda)},
\end{equation}
where $\eta = 2\theta$. This finishes our proof.$\hfill\qed$

\section*{Acknowledgments}
  Tomas Va\v{s}kevi\v{c}ius would like to thank Jaouad Mourtada and Nikita
  Zhivotovskiy for many discussions related to high probability excess risk
  bounds.

  Varun Kanade and Patrick Rebeschini are supported in part by the Alan Turing
  Institute under the EPSRC grant EP/N510129/1. Tomas Va\v{s}kevi\v{c}ius is
  supported by the EPSRC and MRC through the OxWaSP CDT programme (EP/L016710/1).

\bibliographystyle{plainnat}
{\footnotesize \bibliography{references}}

\appendix

\section{Deferred Proofs}
\label{appendix:deferred-proofs}

\subsection{Proof of Lemma~\ref{lemma:offset-localization-not-worse}}
\label{sec:proof-of-lemma-offset-not-worse}

  Fix any $\varepsilon > 0$ and let $\lambda = (1 + \varepsilon)^{-1} \in
  (0 ,1)$. Let $\lambda\mathcal{H} =
  \{\lambda h : h \in \mathcal{H}\}$ and observe that by the star-shapedness
  assumption we have $\lambda \mathcal{H} \subseteq \mathcal{H}$. It follows
  that
  \begin{align}
    \label{eq:offset-rad-rescaling}
    \mathfrak{R}^{\text{off}}_{n}(P_{X}, \mathcal{H}, \gamma)
    =
    \lambda^{-1} \mathfrak{R}^{\text{off}}_{n}(P_{X}, \lambda \mathcal{H}, \lambda^{-1} \gamma)
    \leq
    \lambda^{-1}
    \mathfrak{R}^{\text{off}}_{n}(P_{X}, \mathcal{H}, \lambda^{-1} \gamma).
  \end{align}
  We now proceed via a peeling argument.
  For any $r_{1} \geq 0, r_{2} > 0$ denote
  $\mathcal{H}(r_{1}, r_{2}) = \{ h \in \mathcal{H} : \E_{X \sim
  P_{X}}[h(X)^{2}] \in [r_{1}, r_{2}] \}.$
  Denote $\mathfrak{R}_{n}^{\text{loc}} = \mathfrak{R}_{n}^{\text{loc}}(P_{X}, \mathcal{H}, \gamma)$.
  Let $\mathcal{H}_{0} = \mathcal{H}(0, \gamma^{-1} \mathfrak{R}_{n}^{\text{loc}})$ and for
  $k=1, 2,\dots$, let
  $\mathcal{H}_{k} = \mathcal{H}(\lambda^{1-k}\gamma^{-1}\mathfrak{R}_{n}^{\text{loc}},
  \lambda^{-k}\gamma^{-1}\mathfrak{R}_{n}^{\text{loc}}) \cup \{h_{0}\}$, where $h_{0}$ denotes
  the identically zero function.
  Since $\mathcal{H} = \cup_{k \geq 0} \mathcal{H}_{k}$,
  by \eqref{eq:offset-rad-rescaling} we have
  \begin{equation}
    \label{eq:offset-rad-peeling}
    \mathfrak{R}^{\text{off}}_{n}(P_{X}, \mathcal{H}, \gamma)
    \leq \lambda^{-1}
    \sum_{k \geq 0} \mathfrak{R}^{\text{off}}_{n}(P_{X}, \mathcal{H}_{k}, \lambda^{-1}
    \gamma).
  \end{equation}
  Observe that by the definition of $\mathfrak{R}_{n}^{\text{loc}}$ (cf.\
  Definition~\ref{dfn:local-complexity}) we have
  $$
    \mathfrak{R}^{\text{off}}_{n}(P_{X}, \mathcal{H}_{0}, \lambda^{-1}\gamma)
    \leq
    \mathfrak{R}^{\text{off}}_{n}(P_{X}, \mathcal{H}_{0}, 0)
    \leq
    \mathfrak{R}_{n}^{\text{loc}}.
  $$
  At the same time, for any $k \geq 1$ we have $h_{0} \in \mathcal{H}$ and
  hence $\mathfrak{R}^{\text{off}}_{n}(P_{X}, \mathcal{H}_{k}, \lambda^{-1}\gamma) \geq 0$.
  Also, by \citep*[Lemmas 3.2 and 3.4]{bartlett2005local} we have
  $$
    \mathfrak{R}^{\text{off}}_{n}(P_{X}, \mathcal{H}(0,
    \lambda^{-k}\gamma^{-1}\mathfrak{R}_{n}^{\text{loc}}), 0) \leq
    \lambda^{-k}\mathfrak{R}_{n}^{\text{loc}}
  $$ and
  consequently,
  \begin{align*}
    0 &\leq \mathfrak{R}^{\text{off}}_{n}(P_{X}, \mathcal{H}_{k}, \lambda^{-1}\gamma)
    \leq \mathfrak{R}^{\text{off}}_{n}(P_{X}, \mathcal{H}_{k}, 0)
    - \lambda^{-1}\gamma \cdot \lambda^{1-k} \gamma^{-1} \mathfrak{R}_{n}^{\text{loc}}
    \\
    &= \mathfrak{R}^{\text{off}}_{n}(P_{X}, \mathcal{H}_{k}, 0)
    - \lambda^{-k}\mathfrak{R}_{n}^{\text{loc}}
    \leq \mathfrak{R}^{\text{off}}_{n}(P_{X}, \mathcal{H}(0,
    \lambda^{-k}\gamma^{-1}\mathfrak{R}_{n}^{\text{loc}}), 0)
      - \lambda^{-k}\mathfrak{R}_{n}^{\text{loc}}
      \leq 0.
  \end{align*}
  Hence, combining the above display equations,
  the inequality \eqref{eq:offset-rad-peeling} simplifies to
  $$
    \mathfrak{R}_{n}^{\text{off}}
    (P_{X}, \mathcal{H}, \gamma) \leq \lambda^{-1} \mathfrak{R}_{n}^{\text{loc}} = (1 +
    \varepsilon) \mathfrak{R}_{n}^{\text{loc}}.
  $$
  Since the choice of $\varepsilon > 0$ is
  arbitrary, our proof is complete.$\hfill\qed$

\subsection{Proof of Lemma~\ref{lemma:sparse-offset-complexity}}
\label{sec:proof-of-sparsity-lemma}

Let $\Phi \in \R^{n \times d}$ denote a matrix such that
$\Phi_{i,j} = (\Phi_{i})_{j}$ for any $i \in \{1, \dots, n\}$ and $j \in \{1,
\dots, d\}$. To simplify the notation let $\mathcal{F} =
\mathcal{F}_{\text{lin}}^{d,k}$.
For any $S \subseteq \{1, 2, \dots, d\}$, let $\Phi_{S} \in \mathbb{R}^{n
\times |S|}$ denote the matrix obtained by keeping only the columns of $\Phi$
indexed by the set $S$ and let $$\mathcal{S}^{d,k} = \{S \subseteq \{1, \dots,
d\} : |S| \leq k\}.$$ Observe that for any $\lambda > 0$ by Jensen's
inequality, the fact that $x \mapsto e^{\lambda x}$ is increasing, and
replacing maximum by a sum, we have
\begin{align}
  &n \mathfrak{R}^{\text{off}}(S_{n}^{\Phi}, \mathcal{F}, \gamma)
  \\
  &= \E_{\sigma} \sup_{\ip{w}{\cdot} \in \mathcal{F}}
  \left\{
    \sum_{i=1}^{n}\sigma_{i}\ip{w}{\Phi_{i}}
    - \gamma\ip{w}{\Phi_{i}}^{2}
  \right\}
  \\
  &= \E_{\sigma} \sup_{\ip{w}{\cdot} \in \mathcal{F}}
  \left\{
    \ip{\Phi w}{\sigma}
    - \gamma w^{\top}(\Phi^{\top}\Phi)w
  \right\}
  \\
  &=\E_{\sigma}
  \max_{S \in \mathcal{S}^{d,k}}
  \sup_{w \in \mathbb{R}^{|S|}}
  \left\{
    \ip{\Phi_{S} w}{\sigma}
    - \gamma w^{\top}(\Phi^{\top}_{S}\Phi_{S})w
  \right\}
  \\
  &\leq \frac{1}{\lambda} \log
  \E_{\sigma} \exp\left(\lambda
  \max_{S \in \mathcal{S}^{d,k}}
  \sup_{w \in \mathbb{R}^{|S|}}
  \left\{
    \ip{\Phi_{S} w}{\sigma}
    - \gamma w^{\top}(\Phi^{\top}_{S}\Phi_{S})w
  \right\}
  \right)
  \\
  &\leq \frac{1}{\lambda} \log
  \sum_{S \in \mathcal{S}^{d,k}}
  \E_{\sigma} \exp\left(\lambda
  \sup_{w \in \mathbb{R}^{|S|}}
  \left\{
    \ip{\Phi_{S} w}{\sigma}
    - \gamma w^{\top}(\Phi^{\top}_{S}\Phi_{S})w
  \right\}
  \right)
  \\
  &\leq \frac{1}{\lambda} \log \left(
  \left|\mathcal{S}^{d,k}\right|
  \max_{S \in \mathcal{S}^{d,k}}
  \E_{\sigma} \exp\left(\lambda
  \sup_{w \in \mathbb{R}^{|S|}}
  \left\{
    \ip{\Phi_{S} w}{\sigma}
    - \gamma w^{\top}(\Phi^{\top}_{S}\Phi_{S})w
  \right\}
  \right)
  \right).
  \label{eq:sparsity-bound-second-rearrangement}
\end{align}
We now proceed to upper bound the expectation inside the logarithm.
For any matrix $A$, denote its Moore-Penrose inverse by $A^{\dagger}$.
Fix any $S \in \mathcal{S}^{d,k}$.
For any vector $\sigma \in \R^{n}$, the vector $\Phi_{S}^{\top}\sigma$
belongs to the orthogonal complement of the null space of $\Phi_{S}^{\top}\Phi_{S}$.
Hence, following \citep*[Section 12, page 108]{rockafellar1970convex},
the following identity holds:
\begin{align*}
  \sup_{w \in \R^{|S|}}
  \left\{
    \ip{\Phi_{S} w}{\sigma}
    - \gamma w^{\top}(\Phi^{\top}_{S}\Phi_{S})w
  \right\}
  &=
  \sup_{w \in \R^{|S|}}
  \left\{
    \ip{w}{\Phi_{S}^{\top}\sigma}
    - \gamma w^{\top}(\Phi^{\top}_{S}\Phi_{S})w
  \right\}
  \\
  &= (4\gamma)^{-1}\sigma^{\top}
    \Phi_{S}(\Phi_{S}^{\top}\Phi_{S})^{\dagger}
    \Phi_{S}^{\top}\sigma.
\end{align*}
To simplify the notation, denote
by $H  = \Phi_{S}(\Phi_{S}^{\top}\Phi_{S})^{\dagger}\Phi_{S}^{\top}$ the
hat matrix, keeping the dependence on an arbitrary fixed $S \in
\mathcal{S}^{d,k}$ implicit.
By the above equation, it follows that
\begin{equation}
  \E_{\sigma} \exp\left(\lambda
  \sup_{w \in \mathbb{R}^{|S|}}
  \left\{
    \ip{\Phi_{S} w}{\sigma}
    - \gamma w^{\top}(\Phi^{\top}_{S}\Phi_{S})w
  \right\}
  \right)
  =
  \E_{\sigma}
  \exp\left(\frac{\lambda}{4\gamma}\sum_{i,j=1}^{n}\sigma_{i}\sigma_{j}H_{i,j}
  \right).
\end{equation}
We will now control the moment generating function of the
above Rademacher chaos by decoupling and comparison
with Gaussian chaos. Let
$\sigma' = (\sigma'_{1}, \dots, \sigma'_{n})^{\top}$ be an
independent copy of $\sigma$. Let $g = (g_{1}, \dots, g_{n})^{\top} \in
\mathbb{R}^{n}$ be a vector of independent standard Normal random variables
and let $g'$ be an independent copy of $g$. Then, for some universal constant $c_{1} >
0$ we have
\begin{align*}
  &\E_{\sigma}
  \exp\left(\frac{\lambda}{4\gamma}\sum_{i,j=1}^{n}\sigma_{i}\sigma_{j}H_{i,j}
  \right)
  \\
  &\leq
  \E_{\sigma, \sigma'}
  \exp\left(\frac{\lambda}{\gamma}\sum_{i,j=1}^{n}\sigma_{i}\sigma'_{j}H_{i,j}
  \right)
  &\text{\citep*[(Decoupling) Theorem 6.1.1]{vershynin2018high}}
  \\
  &\leq
  \E_{g, g'}
  \exp\left(\frac{c_{1}\lambda}{\gamma}\sum_{i,j=1}^{n}g_{i}g'_{j}H_{i,j}
  \right)
  &\text{\citep*[(Comparison) Lemma 6.2.3]{vershynin2018high}}.
\end{align*}
Let $\|\cdot\|_{\text{op}}$ denote the operator norm and let $\|\cdot\|_{F}$
denote the Frobenius norm. Then, by the Gaussian chaos moment generating function
bound \citep*[Lemma 6.2.2]{vershynin2018high}, there exist some universal
constants $c_{2},c_{3} > 0$ such that for any $\lambda \in (0,
\gamma c_{2}/\|H\|_{\text{op}}]$
we have
\begin{equation}
  \label{eq:final-chaos-bound}
  \E_{g, g'}
  \exp\left(\frac{c_{1}\lambda}{\gamma}\sum_{i,j=1}^{n}g_{i}g'_{j}H_{i,j}
  \right)
  \leq
  \exp\left(\frac{c_{3}\lambda^{2}}{\gamma^{2}}\|H\|_{F}^{2}
  \right).
\end{equation}
We will now plug in the above bound into
\eqref{eq:sparsity-bound-second-rearrangement}. Notice that the hat matrix
$H$ has at most $|S|$ non-zero eigenvalues, all of which are equal to $1$;
hence,$\|H\|_{op} = 1$ and $\|H\|_{F}^{2} \leq |S|$.
It follows that for any $\lambda \in (0, \gamma c_{2}]$ we have
\begin{align}
  \label{eq:rademacher-bound-without-optimized-lambda}
  \E_{\sigma} \sup_{w \in \mathbb{R}^{d}, \|w\|_{0} \leq k}
  \left\{
    \ip{\Phi w}{\sigma}
    - \gamma w^{\top}(\Phi^{\top}\Phi)w
  \right\}
  \leq
  \frac{1}{\lambda}\log | \mathcal{S}^{d,k} |
  + \frac{c_{3}\lambda k}{\gamma^{2}}.
\end{align}
Recalling the standard bound
$$
  |\mathcal{S}^{d,k}| = \sum_{i=1}^{k} \binom{d}{i}
  \leq \left(\frac{ed}{k}\right)^{k}
$$
and plugging in $\lambda = \gamma c_{2}$ in
\eqref{eq:rademacher-bound-without-optimized-lambda}
yields the desired result
$$
  n \mathfrak{R}^{\text{off}}(S_{n}^{\Phi}, \mathcal{F}, \gamma)
  \leq
  \frac{1}{\gamma}\left(
    c_{2}^{-1}k\log\frac{ed}{k}
    + c_{2}c_{3}k
  \right)
  \lesssim
  \frac{1}{\gamma}\log\left(\frac{ed}{k}\right)k.
$$
$\hfill\qed$

\subsection{Proof of Lemma~\ref{lemma:midpoint-estimator-is-offset}}
\label{sec:proof-of-lemma-midpoint-estimator-is-offset}

  For any $g, g' \in \mathcal{G}$ define the event
  \begin{align*}
    E(g,g') = \bigg\{
      R(g) - R(g')
      \leq R_{n}(g) - R_{n}(g') + c_{1}C_{b}d_{\delta, n}(g, g')
    \bigg\}.
  \end{align*}
  By the empirical Bernstein inequality \citep*[Thereom 11]{maurer2009empirical}
  applied to the random variables
  $(2bC_{b})^{-1}(\ell_{g}(X_{i}, Y_{i}) - \ell_{g'}(X_{i}, Y_{i}))$
  we have $\P(E(g,g')) \geq 1 - \delta/m^{2}$. Hence, defining the event $E =
  \cup_{g,g'\in \mathcal{G}} E(g,g')$, by the union bound $\P(E) \geq
  1-\delta$.

  We will now show that on the event $E$, the estimator
  $\widehat{f}^{\text{(mid)}}$ satisfies the offset condition.
  First observe that on the event
  $E(\widehat{f}^{\text{(ERM)}}, g^{\star}) \subseteq E$,
  the population risk minimizer $g^{\star}$ belongs to the set
  $\mathcal{G}_{\delta, c_{1}}(S_{n})$ of the empirical almost minimizers.
  Define the diameter
  $$
    D_{n}^{\text{max}} = \max_{g,g' \in \mathcal{G}_{\delta, c_{1}}(S_{n})}
    \|g - g'\|_{n}^{2},\quad\text{where}\quad
    \|g-g'\|_{n}^{2} = \frac{1}{n}\sum_{i=1}^{n}(g(X_{i}) - g'(X_{i}))^{2}.
  $$
  We may assume without loss of generality that $D_{n}^{\text{max}} > 0$ since
  otherwise the offset condition is trivially satisfied.
  Since $g^{\star} \in \mathcal{G}_{\delta, c_{1}}(S_{n})$, it follows that
  $\|\widehat{f}^{\text{(mid)}} - g^{\star}\|_{n}^{2} \leq D_{n}^{\text{max}}$.
  Also, since $D_{n}^{\text{max}} > 0$, there exists some function $g' \in
  \mathcal{G}_{\delta, c_{1}}(S_{n})$ such that $\|\widehat{f}^{\text{(ERM)}} - g'\|
  \geq D_{n}^{\text{max}}/4$. Hence, on the event $E$ it holds that
  \begin{align*}
    &R_{n}(\widehat{f}^{\text{(mid)}}) - R_{n}(g^{\star})
    \\
    &\leq
    R_{n}\left(\frac{\widehat{f}^{\text{(ERM)}} + g'}{2}\right) - R_{n}(g^{\star})
    \\
    &\leq
    \frac{1}{2}(R_{n}(\widehat{f}^{\text{(ERM)}}) - R_{n}(g^{\star}))
    + \frac{1}{2}(R_{n}(g') - R_{n}(g^{\star}))
    - \frac{\gamma}{32}D_{n}^{\text{max}},
    \\
    &\leq
    \left(\frac{1}{2}c_{1}C_{b}\sqrt{\frac{D_{n}^{\text{max}}\log(2m/\delta)}{n}}
    - \frac{\gamma}{64}D_{n}^{\text{max}}\right)
    + \frac{1}{2}c_{1}bC_{b}\frac{\log(2m/\delta)}{n}
    - \frac{\gamma}{64}D_{n}^{\text{max}},
    \\
    &\leq
    \left(4c_{1}^{2}C_{b}^{2}\gamma^{-1} + \frac{1}{2}c_{1}bC_{b}\right)
    \frac{\log(2m/\delta)}{n} - \frac{\gamma}{64}\|\widehat{f}^{\text{(mid)}} -
    g^{\star}\|_{n}^{2},
  \end{align*}
  where the third line follows by the strong convexity of the loss function;
  the fourth line follows by the fact that $g' \in \mathcal{G}_{\delta,
  c_{1}}(S_{n})$ and $R_{n}(\widehat{f}^{\text{(ERM)}}) - R_{n}(g^{\star}) \leq
  0$; the fifth line follows by
  optimizing the quadratic function in
  $\sqrt{D_{n}^{\text{max}}}$ in the brackets and replacing
  $D_{n}^{\text{max}}$ by $\|\widehat{f}^{\text{(mid)}} - g^{\star}\|_{n}^{2}$.
  By Remark~\ref{remark:lipschitz-parameter-size}, we have $bC_{b} \leq
  \gamma^{-1}C_{b}^{2}$ and thus our proof is complete.$\hfill\qed$

\end{document}